\documentclass{amsart}
\usepackage{graphicx}
\usepackage{amsmath}
\usepackage{amsthm}
\usepackage{amssymb,bbm}%
\usepackage[numbers, square]{natbib}

\newcommand{\E}{\mathbb E}
\newcommand{\R}{\mathbb{R}}
\newcommand{\N}{\mathbb{N}}
\newcommand{\C}{\mathbb{C}}
\newcommand{\Z}{\mathbb{Z}}

\renewcommand{\P}{\mathbb{P}}

\newcommand{\HHH}{\mathbf {H}}
\newcommand{\QQQ}{\mathbf {Q}}
\newcommand{\KKK}{\mathbf {K}}
\newcommand{\GGG}{\mathbf {G}}
\newcommand{\WWW}{\mathbf {W}}
\newcommand{\FFF}{\mathbf {F}}
\newcommand{\DDD}{\mathbb{D}}
\newcommand{\TTT}{\mathbb{T}}
\newcommand{\JJJ}{\mathcal{J}}

\newcommand{\MMM}{\mathfrak{M}}

\newcommand{\eps}{\varepsilon}

\newcommand{\toprobab}{\overset{P}{\underset{n\to\infty}\longrightarrow}}

\newcommand{\toas}{\overset{a.s.}{\underset{n\to\infty}\longrightarrow}}
\newcommand{\ton}{\overset{}{\underset{n\to\infty}\longrightarrow}}

\newcommand{\ind}{\mathbbm{1}}
\newcommand{\limn}{\lim_{n\to\infty}}

\renewcommand\Re{\operatorname{Re}}

\theoremstyle{plain}
\newtheorem{theorem}{Theorem}[section]
\newtheorem{lemma}[theorem]{Lemma}
\newtheorem{corollary}[theorem]{Corollary}
\newtheorem{proposition}[theorem]{Proposition}

\theoremstyle{definition}

\theoremstyle{remark}
\newtheorem{remark}[theorem]{Remark}

\newtheorem{example}[theorem]{Example}

\begin{document}

\author{Zakhar Kabluchko}
\address{Zakhar Kabluchko, Institute of Stochastics,
Ulm University,
Helmholtzstr.\ 18,
89069 Ulm, Germany}
\email{zakhar.kabluchko@uni-ulm.de}
\author{Dmitry Zaporozhets}
\address{Dmitry Zaporozhets\\
St.\ Petersburg Branch of the
Steklov Institute of Mathematics,
Fontanka Str.\ 27,
 191011 St.\ Petersburg,
Russia}
\email{zap1979@gmail.com}
\title[Universality for zeros of random analytic functions]{Universality for zeros of random analytic functions}

\begin{abstract}
Let $\xi_0,\xi_1,\ldots$ be independent identically distributed (i.i.d.) random variables such that $\E \log (1+|\xi_0|)<\infty$. We consider random analytic functions of the form
$$
\GGG_n(z)=\sum_{k=0}^{\infty} \xi_k f_{k,n} z^k, 
$$
where $f_{k,n}$ are deterministic complex coefficients.
Let $\nu_n$
be the random measure assigning the same weight $1/n$ to each complex zero of $\GGG_n$. Assuming essentially that  $-\frac 1n \log f_{[tn], n}\to u(t)$ as $n\to\infty$, where $u(t)$ is some function,  we show that the measure $\nu_n$ converges weakly to some deterministic  measure which is characterized in terms of the Legendre--Fenchel transform of $u$. The limiting measure is universal, that is it does not depend on the distribution of the $\xi_k$'s.  This result is applied to several ensembles of random analytic functions including the ensembles corresponding to the three two-dimensional geometries of constant curvature. As another application, we prove a random polynomial analogue of the circular law for random matrices.
\end{abstract}
\maketitle

\section{Introduction}\label{sec:intro}
\subsection{Statement of the problem}\label{subsec:intro}
The simplest ensemble of random polynomials are the \textit{Kac polynomials} defined as
$$
\KKK_n(z)=\sum_{k=0}^n \xi_k z^k,
$$
where $\xi_0,\xi_1,\ldots$ are non-degenerate independent identically distributed (i.i.d.) random variables. The distribution of zeros of Kac polynomials has been much studied; see~\cite{hammersley,sparo_sur,arnold,shepp_vanderbei,ibr_zeit,shmerling_hochberg,hughes_nikeghbali,iz_log}. It is known that under a very mild moment condition, the complex zeros of $\KKK_n$ cluster asymptotically near the unit circle $\TTT=\{|z|=1\}$ and that the distribution of zeros is asymptotically uniform with regard to the argument. Given an analytic function $G$ which does not vanish identically,  we consider a measure $\mu_{G}$ counting the complex zeros of $G$ according to their multiplicities:
$$
\mu_{G}=\sum_{z\in \C: G(z)=0} n_G(z)\delta (z).
$$
Here, $n_{G}(z)$ is the multiplicity of the zero at $z$ and  $\delta(z)$ is the unit point mass at $z$. Then, \citet{iz_log} proved that the following two conditions are equivalent:
\begin{enumerate}
\item With probability $1$, the sequence of measures $\frac 1n \mu_{\KKK_n}$
converges as $n\to\infty$ weakly to the uniform probability distribution on $\TTT$.
\item $\E \log(1+|\xi_0|)<\infty$.
\end{enumerate}

Along with the Kac polynomials, many other remarkable ensembles of random polynomials (or, more generally, random analytic functions) appeared in the literature. These ensembles are usually characterized by  invariance properties with respect to certain groups of transformations and have the general form
$$
\GGG_n(z)=\sum_{k=0}^{\infty} \xi_k f_{k,n}z^k,
$$
where $\xi_0,\xi_1,\ldots$ are i.i.d.\ random variables and $f_{k,n}$ are complex deterministic coefficients.
The aim of the present work is to study the distribution of zeros of $\GGG_n$ asymptotically as $n\to\infty$. More precisely, we will show that under broad assumptions on the coefficients $f_{k,n}$, the random measure $\frac 1n \mu_{\GGG_n}$ converges, as $n\to\infty$, to some limiting deterministic measure $\mu$. The limiting measure $\mu$ is \textit{universal}, that is it does not depend on the distribution of the random variables $\xi_k$; see Figure~\ref{fig:weyl_universality}.
Universality has been much studied in the context of random matrices; see, e.g., \cite{tao_vu}. The literature on random polynomials and random analytic functions usually concentrates on the Gaussian case, since in this case explicit calculations are possible; see, e.g., \cite{hammersley,edelman_kostlan,peres_book,sodin_tsirelson,shepp_vanderbei,schiffman_zelditch_bundles,shiffman_zelditch1,bloom_shiffman,sodin_ECM,farahmand_book,bharucha_reid_book}. The only ensemble of random polynomials for which universality is well-understood is the Kac ensemble; see~\cite{sparo_sur,arnold,ibr_zeit,iz_log}.
In the context of random polynomials there have been a number of results on the \textit{local universality} in the distribution of zeros~\cite{bleher_di,ledoan_etal,schiffman_zelditch_bundles,shiffman_zelditch1}. For example, universal character of local correlations between close zeros has been demonstrated for some models. In this work, our focus is different: we prove the universality of the distribution of zeros on the \textit{global scale}.



The paper is organized as follows. In Sections~\ref{subsec:invariant}--\ref{subsec:theta} we state our results for a number of concrete ensembles of random analytic functions. These results are special cases of the general Theorem~\ref{theo:general} whose statement, due to its technicality,  is postponed to Section~\ref{subsec:general_theo}. Proofs are given in Sections~\ref{sec:proofs_special_cases}, \ref{sec:proof_log_asympt}, \ref{sec:proof_general}.

\subsection{Notation}\label{subsec:notation}
Let $\DDD_r=\{z\in\C: |z|<r\}$ be the open disk with radius $r>0$ centered at the origin. Let $\DDD=\DDD_1$ be the unit disk.  Denote by $\lambda$ the Lebesgue measure on $\C$.
A Borel measure $\mu$ on a Polish space  $X$ is called \textit{locally finite} (l.f.) if $\mu(A)<\infty$ for every compact set $A\subset X$. 
A sequence $\mu_n$ of l.f.\ measures on $X$ converges \textit{vaguely} to a l.f.\ measure $\mu$ if for every continuous, compactly supported  function $\varphi:X\to\R$,
\begin{equation}\label{eq:vague_conv_meas}
\int_{X} \varphi(z) \mu_n(dz) \ton \int_{X} \varphi(z) \mu(dz).
\end{equation}
If $\mu_n$ and $\mu$ are probability measures, the vague convergence is equivalent to the more familiar weak convergence for which~\eqref{eq:vague_conv_meas} is required to hold for all continuous, bounded $\varphi$.
Let $\MMM(X)$ be the space of all l.f.\ measures on $X$ endowed with the vague topology. A \textit{random measure} on $X$ is a random element with values in $\MMM(X)$.
A sequence of random measures $\mu_n$ converges to a random measure $\mu$ \textit{in probability} (respectively, \textit{a.s.}), if~\eqref{eq:vague_conv_meas} holds in probability (respectively, a.s.) for every continuous, compactly supported function $\varphi$.

\begin{figure}[t]
\includegraphics[height=0.49\textwidth, width=0.49\textwidth]{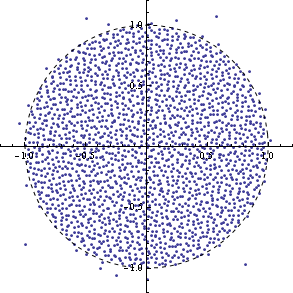}
\includegraphics[height=0.49\textwidth, width=0.49\textwidth]{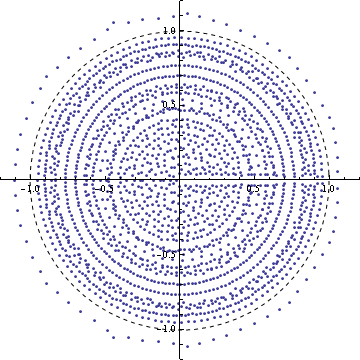}
\caption
{Zeros of the Weyl random polynomial $\WWW_n(z)=\sum_{k=0}^n \xi_k \frac{z^k}{\sqrt {k!}}$ of degree $n=2000$. The zeros were divided by $\sqrt n$. Left: complex normal coefficients.  Right: coefficients are real with $\P[\log |\xi_k|>t]=t^{-4}$ for $t>1$.
In both cases, the limiting distribution of zeros is uniform on the unit disk.
}\label{fig:weyl_universality}
\end{figure}

\section{Statement of results}\label{sec:statement_of_results}
\subsection{The three invariant ensembles}\label{subsec:invariant}
Let $\xi_0,\xi_1,\ldots$ be non-degenerate i.i.d.\ random variables. 
Fix a parameter $\alpha>0$.
We start by considering the following three ensembles of random analytic functions (see, e.g., \cite{sodin_tsirelson,peres_book} and Figure~\ref{fig:LO_roots}):
\begin{equation*}
\FFF_n(z)
=
\begin{cases}
\sum_{k=0}^n \xi_k \left(\frac{n(n-1)\ldots (n-k+1)}{k!}\right)^{\alpha}  z^k,
&(\text{elliptic, } n\in\N, z\in\C),\\
\sum_{k=0}^{\infty} \xi_k \left(\frac{n^k}{k!}\right)^{\alpha} z^k,
&(\text{flat, } n>0, z\in\C),\\
\sum_{k=0}^{\infty} \xi_k \left(\frac{n(n+1)\ldots (n+k-1)}{k!}\right)^{\alpha} z^k,
&(\text{hyperbolic, } n>0, z\in\DDD).
\end{cases}
\end{equation*}
Note that in the elliptic case $\FFF_n$ is a random polynomial of degree $n$, in the flat case it is a random entire function, whereas in the hyperbolic case it is a random analytic function defined on the unit disk $\DDD$.

In the particular case when $\alpha=1/2$ and $\xi_k$ are complex standard Gaussian with density $z\mapsto \pi^{-1} \exp\{-|z|^2\}$ on $\C$, the zero sets of these analytic functions possess remarkable invariance properties~\cite{sodin_tsirelson,peres_book}. Namely, in the flat case, the law of the zero set of $\FFF_n$ is invariant with respect to the rigid motions of $\C$. In the elliptic (resp., hyperbolic) case, the law of the zero set of $\FFF_n$ is invariant with respect to the isometries of the Riemann sphere $\bar \C$ (resp., the unit disk $\DDD$) preserving the spherical metric of constant positive curvature (resp., the hyperbolic metric of constant negative curvature).
Thus, the three ensembles correspond to the three two-dimensional geometries of constant curvature.
If $\alpha=1/2$ and $\xi_k$ are complex Gaussian, the expected number of zeros of $\FFF_n$ in a Borel set $B$ can be computed exactly:
\begin{equation*}
\E [\mu_{\FFF_n}(B)]
=
\begin{cases}
\frac n{\pi} \int_B (1+ |z|^{2})^{-2} \lambda(dz),
& (\text{elliptic case, } B\subset \C), \\
\frac n{\pi} \lambda(B),
& (\text{flat case, } B\subset \C), \\
\frac n{\pi} \int_B (1-|z|^{2})^{-2}\lambda(dz),
& (\text{hyperbolic case, } B\subset \DDD).
\end{cases}
\end{equation*}
The next theorem states the universality for the distribution of zeros of $\FFF_n$.
\begin{theorem}\label{theo:invariant}
Let $\xi_0,\xi_1,\ldots$ be non-degenerate i.i.d.\ random variables such that $\E \log (1+|\xi_0|)<\infty$. As $n\to\infty$, the  sequence of random measures $\frac 1n \mu_{\FFF_n}$ converges in probability to the deterministic  measure having a density $\rho_{\alpha}$ with respect to the Lebesgue measure, where
\begin{equation*}
\rho_{\alpha}(z)
=
\begin{cases}
\frac 1{2\pi\alpha}|z|^{\frac 1 {\alpha} - 2} (1+ |z|^{\frac 1{\alpha}})^{-2},
&(\text{elliptic case, } z\in\C),\\
\frac 1{2\pi \alpha} |z|^{\frac 1 {\alpha}-2},
&(\text{flat case, } z\in\C),\\
\frac 1 {2\pi \alpha} |z|^{\frac 1 {\alpha} - 2} (1-|z|^{\frac 1 {\alpha}})^{-2},
&(\text{hyperbolic case, } z\in\DDD).
\end{cases}
\end{equation*}
\end{theorem}

\begin{figure}[t]
\includegraphics[height=0.49\textwidth, width=0.49\textwidth]{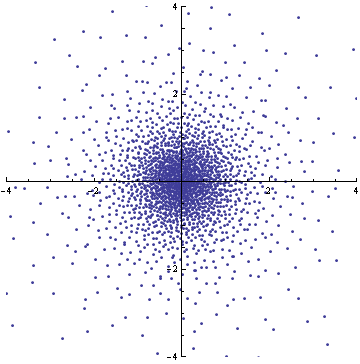}
\includegraphics[height=0.49\textwidth, width=0.49\textwidth]{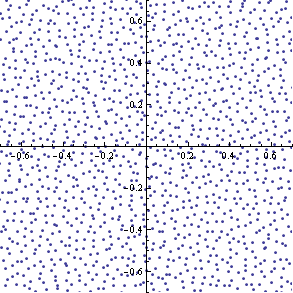}
\caption
{
Zeros in the elliptic (left) and flat (right) models with $\alpha=1/2$ and $n=2000$. The coefficients are complex normal.
} \label{fig:LO_roots}
\end{figure}

\begin{figure}[t]
\includegraphics[height=0.49\textwidth, width=0.49\textwidth]{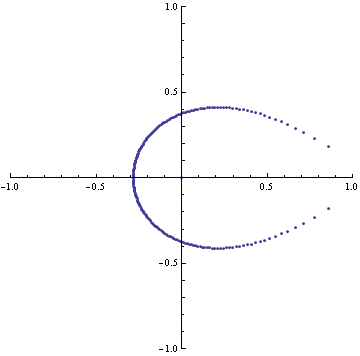}
\includegraphics[height=0.49\textwidth, width=0.49\textwidth]{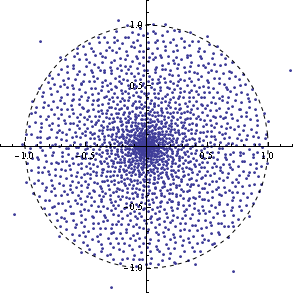}
\caption
{Left: Zeros of the Szeg\H{o} polynomial $s_n(z)=\sum_{k=0}^n \frac {z^k}{k!}$ of degree $n=200$.  Right: Zeros of the Littlewood--Offord random polynomial $\WWW_n(z)=\sum_{k=0}^n \xi_k \frac{z^k}{k!}$ of degree $n=2000$ with complex normal coefficients. The zeros were divided by $n$.
}\label{fig:szegoe}
\end{figure}

\subsection{Littlewood--Offord random polynomials}\label{subsec:LO_poly}
Next we consider an ensemble of random polynomials which was introduced by~\citet{littlewood_offord1,littlewood_offord2}. It is related to the flat model. First we give some motivation.
Let $\xi_0,\xi_1,\ldots$ be non-degenerate i.i.d.\ random variables.
Given a complex sequence $w_0,w_1,\ldots$ consider a random polynomial $\WWW_n$ defined by
\begin{equation}\label{eq:def_Qn_LO}
\WWW_n(z)=\sum_{k=0}^n \xi_k w_k  z^k.
\end{equation}
For $w_k=1$ we recover the Kac polynomials, for which the zeros concentrate near the unit circle. The next result shows that the structure of the zeros does not differ essentially from the Kac case if  the sequence $w_k$ grows or decays not too fast.
\begin{theorem}\label{theo:kac_generalized}
Let $\xi_0,\xi_1,\ldots$ be non-degenerate i.i.d.\ random variables such that $\E \log (1+|\xi_0|)<\infty$.
If $\lim_{k\to\infty} \frac 1k \log  |w_k| = w$ for some constant $w\in\R$, then the sequence of random measures $\frac 1n \mu_{\WWW_n}$ converges in probability to the uniform probability distribution on the circle of radius $e^{-w}$ centered at the origin.
\end{theorem}

We would like to construct examples where there is no concentration near a circle.
Clearly, the sequence $\log |w_k|$ has to grow or decay superlinearly. We will consider decaying sequences, since for growing sequences the zeros concentrate near the origin.  At a first attempt, it is natural to look at the case in which  $\log |w_k|$ is a multiple of $-k\log k$. More precisely, we make the following assumption on the sequence $w_k$:
\begin{equation}\label{eq:fk_asympt}
\log  |w_k| = - \alpha (k\log k - k) - \beta k +o(k), \qquad k\to\infty,
\end{equation}
where $\alpha>0$ and $\beta\in\R$ are parameters. Particular cases are polynomials of the form
$$
\WWW_n^{(1)}(z)=\sum_{k=0}^n  \frac{\xi_k}{(k!)^{\alpha}} z^k
,\;\;\;
\WWW_n^{(2)}(z)=\sum_{k=0}^n  \frac{\xi_k}{k^{\alpha k}} z^k
,\;\;\;
\WWW_n^{(3)}(z)=\sum_{k=0}^n  \frac{\xi_k}{\Gamma(\alpha k+1)} z^k.
$$
The family $\WWW_n^{(1)}$ has been studied by~\citet{littlewood_offord1,littlewood_offord2} in one of the earliest works on random polynomials. They were interested in the number of real zeros. 
In the next theorem we describe the limiting distribution of complex zeros of $\WWW_n$.
\begin{theorem}\label{theo:LO_poly}
Let $\xi_0,\xi_1,\ldots$ be non-degenerate i.i.d.\ random variables such that $\E \log (1+|\xi_0|)<\infty$. Let $w_0,w_1,\ldots$ be a complex sequence satisfying~\eqref{eq:fk_asympt}. With probability $1$,  the sequence of random  measures $\frac 1n \mu_{\WWW_n}(e^{\beta} n^{\alpha}\cdot)$
converges to the deterministic probability measure having the density
\begin{equation}\label{eq:LO_poly_density}
z\mapsto \frac 1{2\pi \alpha} |z|^{\frac 1 {\alpha}-2}\ind_{z\in\DDD}
\end{equation}
with respect to the Lebesgue measure on $\C$.
\end{theorem}
\begin{remark}
The measure $\mu_{\WWW_n}(e^{\beta} n^{\alpha}\cdot)$ counts points of the form $\frac {z}{e^{\beta} n^{\alpha}}$, where $z$ is a zero of $\WWW_n$.
\end{remark}
For the so-called \textit{Weyl random polynomials} $\WWW_n(z)=\sum_{k=0}^n \xi_k \frac{z^k}{\sqrt {k!}}$ having $\alpha=1/2$ and $\beta=0$, the limiting distribution is uniform on $\DDD$; see Figure~\ref{fig:weyl_universality}. This result can be seen as an analogue of the famous circular law for the distribution of eigenvalues of the non-Hermitian random matrices with i.i.d.\ entries~\cite{tao_vu,bordenave_chafai}.  \citet{forrester_honner} stated the circular law  for Weyl polynomials and discussed the differences and similarities between the matrix and the polynomial cases; see also~\cite{krishnapur_virag}. The analogy between the random matrices with i.i.d.\ entries and the Weyl polynomials is not merely the coincidence of the limiting distributions. Both models are closely connected to the logarithmic potential theory on $\C$ with external field $\frac 12 |z|^2$; see Section~\ref{subsec:log_potential} for more details.


Under a minor additional assumption on the coefficients $w_k$ we can prove that the logarithmic moment condition is not only sufficient, but also necessary for the a.s.\ convergence of the empirical distribution of zeros. It is easy to check that the additional assumption is satisfied for $\WWW_n=\WWW_n^{(i)}$ with $i=1,2,3$.
\begin{theorem}\label{theo:LO_poly_converse}
Let $\xi_0,\xi_1,\ldots$ be non-degenerate i.i.d.\ random variables. Let $w_0,w_1,\ldots$ be a complex sequence satisfying~\eqref{eq:fk_asympt}  and  such that for some $C>0$,
\begin{equation}\label{eq:LO_additional_assumpt}
|w_{n-k}/w_n|< C e^{\beta k} n^{\alpha k} \text{ for all } n\in\N, k\leq n.
\end{equation}
Then, the following are equivalent:
\begin{enumerate}
\item With probability $1$, the sequence of random measures $\frac 1n \mu_{\WWW_n}(e^{\beta}n^{\alpha}\cdot)$
converges to the probability measure on $\DDD$ with density~\eqref{eq:LO_poly_density}.
\item $\E \log (1+|\xi_0|)<\infty$.
\end{enumerate}
\end{theorem}

It should be stressed that in all our results we assume that the random variables $\xi_k$ are non-degenerate (that is, not a.s.\ constant). To see that this assumption is essential, consider the deterministic polynomials
\begin{equation}\label{eq:szegoe_poly}
s_n(z)=\sum_{k=0}^n \frac {z^k}{k!}.
\end{equation}
A classical result of \citet{szegoe} states that the zeros of $s_n(nz)$ cluster asymptotically (as $n\to\infty$) along the curve  $\{|ze^{1-z}|=1\}\cap \DDD$; see Figure~\ref{fig:szegoe} (left). This behavior is manifestly different from the distribution with density $1/(2\pi |z|)$ on $\DDD$ we have obtained in Theorem~\ref{theo:LO_poly} for the same polynomial with randomized coefficients; see Figure~\ref{fig:szegoe} (right).

\subsection{Littlewood--Offord random entire function}\label{subsec:LO_entire}
Next we discuss a random entire function which also was introduced by~\citet{littlewood_offord_entire1,littlewood_offord_entire2err}. Their aim was to describe the properties of a ``typical'' entire function of a given order $1/\alpha$. 
Given a complex sequence $w_0,w_1,\ldots$ satisfying~\eqref{eq:fk_asympt} consider a random entire function
\begin{equation}\label{eq:def_LO_entire}
\WWW(z)=\sum_{k=0}^{\infty} \xi_k w_k  z^k.
\end{equation}
Examples are given by
$$
\WWW^{(1)}(z)=\sum_{k=0}^{\infty}  \frac{\xi_k}{(k!)^{\alpha}} z^k
,\;\;\;
\WWW^{(2)}(z)=\sum_{k=0}^{\infty}  \frac{\xi_k}{k^{\alpha k}} z^k
,\;\;\;
\WWW^{(3)}(z)=\sum_{k=0}^{\infty}  \frac{\xi_k}{\Gamma(\alpha k+1)} z^k.
$$
The first function is essentially the flat model considered above, namely $\WWW(n^{\alpha}z)=\FFF_n(z)$. For $\alpha=1$ it is a randomized version of the Taylor series for the exponential.  The last function is a randomized version of the Mittag--Leffler function.
Our aim is to describe the density of zeros of $\WWW$ on the global scale.
We have the following strengthening of the flat case of Theorem~\ref{theo:invariant}.
\begin{theorem}\label{theo:LO_entire}
Let $\xi_0,\xi_1,\ldots$ be non-degenerate i.i.d.\ random variables such that $\E \log (1+|\xi_0|)<\infty$. Let $w_0,w_1,\ldots$ be a complex sequence satisfying~\eqref{eq:fk_asympt}. With probability $1$, the random measure
$\frac 1n \mu_{\WWW}(e^{\beta} n^{\alpha}\cdot)$ converges to the deterministic measure having the density
\begin{equation}\label{eq:LO_entire_density}
z\mapsto \frac 1{2\pi \alpha} |z|^{\frac 1 {\alpha}-2}
\end{equation}
with respect to the Lebesgue measure on $\C$.
\end{theorem}
As a corollary we obtain a law of large numbers for the number of zeros of $\WWW$.
\begin{corollary}\label{cor:LO_entire_LLN}
Let $N(r)=\mu_{\WWW}(\DDD_r)$ be the number of zeros of $\WWW$ in the disk $\DDD_r$.  Under the assumptions of Theorem~\ref{theo:LO_entire},
$$
N(r) = e^{-\frac{\beta}{\alpha}} r^{\frac 1{\alpha}} (1+o(1)) \text{ a.s. as }  r\to\infty.
$$
\end{corollary}
In the case $\alpha=1/2$ the limiting measure in Theorem~\ref{theo:LO_entire} has constant density $1/\pi$. The difference between the limiting densities in Theorem~\ref{theo:LO_poly} and Theorem~\ref{theo:LO_entire} is that in the latter case there is no restriction to the unit circle.  In the case of Bernoulli-distributed $\xi_k$ the function~\eqref{eq:def_LO_entire} has been considered by~\citet{littlewood_offord_entire1,littlewood_offord_entire2err}. Under the assumption $\log |w_k| \sim -\alpha k \log k$ as $k\to\infty$ they proved some estimates for the distribution of zeros of $\WWW$. These estimates does not seem to imply Theorem~\ref{theo:LO_entire}, since they are true up to multiplicative constants only. Also, the estimates of Littlewood and Offord are stated in terms of some more advanced functions counting the zeros and are not easily translated to our setting.


Let us again stress the importance of the non-degeneracy assumption. The exponential function $e^z$ has no complex zeros, whereas the zeros of its randomized version $\sum_{k=0}^{\infty}\xi_k \frac{z^k}{k!}$ have the global-scale density $1/(2\pi |z|)$ on $\C$. For the absolute values of the zeros, the limiting density is constant and equal to $1$ on $(0,\infty)$.

\subsection{Randomized theta function}\label{subsec:theta}
Given a parameter $\alpha\in (0,1)\cup (1,\infty)$ we consider a random analytic function
$$
\HHH_n(z)
=
\begin{cases}
\sum_{k=0}^{\infty} \xi_k e^{n^{1-\alpha} k^{\alpha}} z^k, &(\text{case } \alpha<1, z\in\DDD),\\
\sum_{k=0}^{\infty} \xi_k e^{-n^{1-\alpha} k^{\alpha}} z^k, &(\text{case } \alpha>1, z\in\C).
\end{cases}
$$
\begin{theorem}\label{theo:stable}
Let $\xi_0,\xi_1,\ldots$ be non-degenerate i.i.d.\ random variables such that $\E \log (1+|\xi_0|)<\infty$. As $n\to\infty$, the sequence of random measures $\frac 1n \mu_{\HHH_n}$ converges in probability to the deterministic measure having the density
$$
z \mapsto
\frac{1}{2\pi \alpha|1-\alpha|} \frac 1 {|z|^2} \left|\frac {\log |z|}{\alpha}\right|^{\frac{2-\alpha}{\alpha-1}}
$$
with respect to the Lebesgue measure on $\C$. The density is restricted to  $\DDD$ in the case $\alpha<1$ and to $\C\backslash \DDD$ in the case $\alpha>1$.
\end{theorem}
As the parameter $\alpha$ crosses the value $1$, the zeros of $\HHH_n$ jump from the unit disk $\DDD$ to its complement $\C\backslash \DDD$; see Figure~\ref{fig:3_circles} (left) for the case $\alpha=2$. Note that the case $\alpha=1$ corresponds formally to Kac polynomials for which the zeros are on the boundary of $\DDD$. The special case $\alpha=2$ corresponds to the randomized theta function
\begin{equation}\label{eq:def_theta}
\HHH_n(z)=\sum_{k=0}^{\infty} \xi_k e^{-\frac{k^2} n} z^k.
\end{equation}
The limiting distribution of zeros has the density $\frac 1{4\pi |z|^2}$ on $\C\backslash \DDD$. One can also take the sum in~\eqref{eq:def_theta} over $k\in\Z$ in which case the zeros fill the whole complex plane with the same density.

A similar model, namely the polynomials $\QQQ_n(z)=\sum_{k=0}^n \xi_k e^{-k^{\alpha}} z^k$, where $\alpha>1$, has been considered by~\citet{schehr_majumdar}. Assuming that $\xi_k$ are real-valued they showed that almost all zeros of $\QQQ_n$ become real if $\alpha>2$. In our model, the distribution of the arguments of the zeros remains uniform  for every $\alpha$.

\begin{figure}[t]
\includegraphics[height=0.49\textwidth, width=0.49\textwidth]{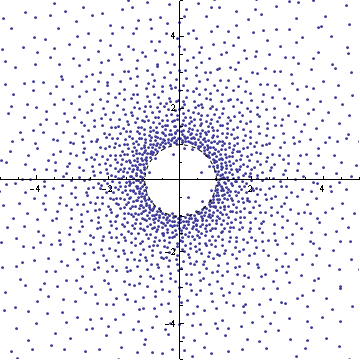}
\includegraphics[height=0.49\textwidth, width=0.49\textwidth]{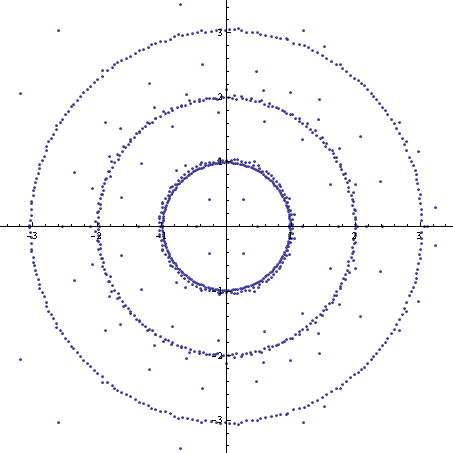}
\caption
{
Left:
Zeros of the randomized theta function $\HHH_n(z)=\sum_{k=0}^{\infty} \xi_k e^{-\frac{k^2}{n}} z^k$ with $n=2000$ and normally distributed coefficients.
Right:
Zeros of the random polynomial from Example~\ref{ex:three_circles} with $n=300$ and Cauchy-distributed coefficients.
}\label{fig:3_circles}
\end{figure}

\subsection{The general result}\label{subsec:general_theo}
We are going to state a theorem which contains all examples considered above as special cases. Let $\xi_0,\xi_1,\ldots$ be non-degenerate i.i.d.\ random variables such that $\E \log(1+|\xi_0|)<\infty$.
Consider a random Taylor series
\begin{equation}\label{eq:def_GGG}
\GGG_n(z)=\sum_{k=0}^{\infty} \xi_k f_{k,n} z^k,
\end{equation}
where $f_{k,n}\in\C$ are deterministic coefficients. We assume that there is a function $f:[0,\infty)\to [0,\infty)$ and a number $T_0\in (0,\infty]$ such that
\begin{enumerate}
\item[(A1)] $f(t)>0$ for $t<T_0$ and $f(t)=0$ for $t>T_0$.
\item[(A2)] $f$ is continuous on $[0,T_0)$, and, in the case $T_0<+\infty$, left continuous at $T_0$.
\item[(A3)]
$
\lim_{n\to\infty} \sup_{k\in[0,An]} \left||f_{k,n}|^{1/n}-f(\frac kn)\right|=0
$
for every $A>0$.
\item[(A4)]
$R_0:=\liminf_{t\to \infty} f(t)^{-1/t}\in(0,\infty]$ and
$\liminf_{n,\frac kn\to \infty} |f_{k,n}|^{-1/k} \geq R_0$.
\end{enumerate}
It will be shown later that Condition~(A4) ensures that the series~\eqref{eq:def_GGG} defining $\GGG_n$ converges with probability $1$ on the disk $\DDD_{R_0}$.
Let $I:\R\to\R\cup\{+\infty\}$ be the Legendre--Fenchel transform of  the function $u(t)=-\log f(t)$, where $\log 0=-\infty$. That is,
\begin{equation}\label{eq:def_I}
I(s)=\sup_{t\geq 0} (st-u(t))=\sup_{t\geq 0} (st+\log f(t)).   
\end{equation}
Note that $I$ is a convex function, $I(s)$ is finite for $s<\log R_0$ and $I(s)=+\infty$ for $s>\log R_0$. Recall that $\mu_{\GGG_n}$  is the measure assigning to each zero of $\GGG_n$ a weight equal to its multiplicity.
\begin{theorem}\label{theo:general}
Under the above assumptions, the sequence of random measures $\frac 1n \mu_{\GGG_n}$ converges in probability to some deterministic locally finite measure $\mu$ on the disk $\DDD_{R_0}$. The measure $\mu$ is rotationally invariant and is characterized by
\begin{equation}\label{eq:theo:general}
\mu(\DDD_r)
=I'(\log r), \;\;\; r<R_0.
\end{equation}
\end{theorem}
By convention, $I'$ is the left derivative of $I$.  Since $I$ is convex, the left derivative exists everywhere on $(-\infty, \log R_0)$ and is a non-decreasing, left-continuous function.
Since the supremum in~\eqref{eq:def_I} is taken over $t\geq 0$, we have $\lim_{s\to-\infty}I'(s)=0$. Hence, $\mu$ has no atom at zero.
If $I'$ is  absolutely continuous on some interval $(\log r_1,\log r_2)$, then the density of $\mu$ on the annulus $r_1<|z|<r_2$ with respect to the Lebesgue measure on $\C$ is
\begin{equation}\label{eq:rho_formula}
\rho(z)=\frac {I''(\log |z|)} {2\pi |z|^2}.
\end{equation}

It is possible to give a characterization of the measure $\mu$ without referring to the Legendre--Fenchel transform. The radial part of $\mu$ is a measure $\bar \mu$  on $(0,\infty)$ defined by $\bar \mu((0,r))=\mu(\DDD_r)$. Suppose first that $u$ is convex on $(0,T_0)$ (which is the case in all our examples). Then, $\bar \mu$ is the image of the Lebesgue measure on $(0,\infty)$ under the mapping $t\mapsto e^{u'(t)}$, where $u'$ is the left derivative of $u$. This follows from the fact that $(u')^{\leftarrow}=I'$ and $(I')^{\leftarrow}=u'$ by the Legendre--Fenchel duality, where $\varphi^{\leftarrow}(t)=\inf\{s\in\R: \varphi(s)\geq t\}$ is the generalized left-continuous inverse of a  non-decreasing function $\varphi$. In particular, the support of $\mu$ is contained in the annulus
$$
\left\{e^{\lim_{t\downarrow 0}u'(t)}\leq |z|\leq e^{\lim_{t\uparrow T_0} u'(t)}\right\}
$$ and is equal to this annulus if $u'$ has no jumps.
In general, any jump of $u'$ (or, by duality, any constancy interval of $I'$) corresponds to a missing annulus in the support of $\mu$. Also, any jump of $I'$ (or, by duality, any constancy interval of $u'$) corresponds to a circle with positive $\mu$-measure.  More precisely, if $I'$ has a jump at $s$ (or, by  duality, $u'$ takes the value $s$ on an interval of positive length), then $\mu$ assigns a positive weight (equal to the size of the jump) to the circle  of radius  $e^s$ centered at the origin.
In the case when $u$ is non-convex we can apply the same considerations after replacing $u$ by its convex hull.

One may ask what measures $\mu$ may appear as limits in Theorem~\ref{theo:general}. Clearly, $\mu$ has to be rotationally invariant, with no atom at $0$. The next theorem shows that there are no further essential  restrictions.
\begin{theorem}\label{theo:general_converse}
Let $\mu$ be a rotationally invariant measure on $\C$  such that
\begin{enumerate}
\item $\mu(\C\backslash \DDD_{R_0})=0$, where $R_0:=\sup\{r>0:\mu(\DDD_r)<\infty\}\in (0,\infty]$.
\item $\int_{0}^{R}\mu(\DDD_r)r^{-1} dr<\infty$ for some (hence, every) $R<R_0$.
\end{enumerate}
Then, there is a random Taylor series $\GGG_n$ of the form~\eqref{eq:def_GGG} with convergence radius a.s.\ $R_0$  such that $\frac 1 n\mu_{\GGG_n}$ converges in probability to $\mu$ on the disk $\DDD_{R_0}$.
\end{theorem}
\begin{example}\label{ex:three_circles}
Consider a random polynomial
\begin{equation}\label{eq:def_three_circles}
\GGG_n(z)=\sum_{k=0}^{n} \xi_k z^k + 2^n\sum_{k=n+1}^{2n} \xi_k \left(\frac {z}{2}\right)^k
+\left(\frac 9 2\right)^n \sum_{k=2n+1}^{3n} \xi_k \left(\frac {z}{3}\right)^k.
\end{equation}
We can apply Theorem~\ref{theo:general} with
$$
u(t)
=
\begin{cases}
0, & t\in[0,1], \\
(\log 2) (t-1), &t\in[1,2],\\
(\log 3) t -\log \frac 92, &t\in[2,3], \\
+\infty, &t\geq 3,
\end{cases}
\;\;\;
I(s)
=
\begin{cases}
0, & s\leq 0,\\
s, & s\in [0,\log 2], \\
2s -\log 2, &s\in[\log 2,\log 3],\\
3s -\log 6, &t\geq \log 3.
\end{cases}
$$
The function $u'$ has three constancy intervals of length $1$ where it takes values $0,\log2,\log 3$. Dually, the function $I'$ has three jumps of size $1$ at $0,\log 2, \log 3$ and is locally constant outside these points. It follows that the limiting distribution of the zeros of $\GGG_n$ is the sum of uniform probability distributions on three concentric circles with radii $1,2,3$; see Figure~\ref{fig:3_circles}.
\end{example}

\begin{remark}
Suppose that $\GGG_n$ satisfies the assumptions of Theorem~\ref{theo:general}. Then, so does the derivative $\GGG_n'$ (and, moreover, $f$ is the same in both cases). Thus, the derivative of any fixed order of $\GGG_n$ has the same limiting distribution of zeros as $\GGG_n$ itself. Similarly, for every complex sequence $c_n$ such that $\lim_{n\to\infty} \frac 1n \log |c_n|\leq f(0)$, the function $\GGG_n(z)-c_n$ satisfies the assumptions. Hence, the limiting distribution of the solutions of the equation $\GGG_n(z)=c_n$ is the same as for the zeros of $\GGG_n$.
\end{remark}


\subsection{Connection with logarithmic potentials and orthogonal polynomials}\label{subsec:log_potential}
The identity $\frac 1 {2\pi}\Delta \log |z-z_0|=\delta(z_0)$  implies the crucial formula
\begin{equation}\label{eq:mu_GGG_is_Laplace_log_GGG}
\mu_{\GGG_n}=\frac 1 {2\pi} \Delta \log |\GGG_n(z)|,
\end{equation}
where the Laplacian $\Delta$ should be understood in the distributional sense.  The first step in the proof of Theorem~\ref{theo:general} is to compute the limiting logarithmic potential in~\eqref{eq:mu_GGG_is_Laplace_log_GGG}.
\begin{theorem}\label{theo:log_asympt}
Under the assumptions of Section~\ref{subsec:general_theo}, for every $z\in\DDD_{R_0}\backslash\{0\}$,
\begin{equation}\label{eq:log_asympt}
p_n(z):=\frac 1n \log |\GGG_n(z)| \toprobab I(\log |z|).
\end{equation}
\end{theorem}
\begin{remark}\label{rem:mu_laplace_I_log}
Taking~\eqref{eq:mu_GGG_is_Laplace_log_GGG} and~\eqref{eq:log_asympt} together we obtain, at least formally,  the following equivalent version of~\eqref{eq:theo:general}:
$\mu=\frac 1{2\pi} \Delta (I(\log|z|))$.
\end{remark}
The measure $\mu$ can be interpreted as an equilibrium measure for the logarithmic potential in the presence of an external field; see~\cite{saff_totik_book} for an account of logarithmic potential theory.  Assume that the assumptions of Section~\ref{subsec:general_theo} hold and additionally,  the function $u$ is  convex on $(0,T_0)$. Since we would like to deal with finite measures,  let us take some $\kappa\in (0,T_0)$ and consider a truncated version of $\GGG_n$:
$$
\GGG_n^{(\kappa)}(z)=\sum_{k=0}^{[\kappa n]} \xi_kf_{k,n} z^k.
$$
If $\GGG_n$ satisfies assumptions~(A1)--(A4) of Section~\ref{subsec:general_theo}, then so does $\GGG_n^{(\kappa)}$ with
$$
u^{(\kappa)}(t)
=
\begin{cases}
u(t), &t\in[0,\kappa]\\
+\infty, &t>\kappa.
\end{cases}
$$
By Theorem~\ref{theo:general}, the measure $\frac 1n \mu_{\GGG_n^{(\kappa)}}$ converges in probability to a limiting measure denoted by $\mu^{(\kappa)}$.
\begin{proposition}\label{prop:equilibrium}
The measure $\mu^{(\kappa)}$ is the unique minimizer of the functional
$$
J(\nu)=\frac 12 \int_{\C}\int_{\C} \log \frac {1}{|z-w|}\nu(dz)\nu(dw) + \int_{\C} I(\log |z|) \nu(dz)
$$
in the set of all measures $\nu$ on $\C$ which have total mass $\kappa$.
\end{proposition}
Denote by $S^{(\kappa)}$ the support of the measure $\mu^{(\kappa)}$.   From the interpretation of $\mu^{(\kappa)}$ as the image of the Lebesgue measure under the mapping $t\mapsto \exp\{(u^{(\kappa)})'(t)\}$ the following monotonicity property follows:  if $\kappa_1<\kappa_2$, then $S^{(\kappa_1)}\subset S^{(\kappa_2)}$.  If the left derivative of $u^{(\kappa)}$  has no constancy intervals, then the measures $\mu^{(\kappa_1)}$ and $\mu^{(\kappa_2)}$ coincide on $S^{(\kappa_1)}$. Thus, as $\kappa$ grows, the zeros fill larger and larger domains without changing their density on domains which are already occupied by the zeros.

Let us finally mention a connection between the random polynomial models considered above and orthogonal polynomials. Let $m$ be a finite measure on $\C$ having an infinite support and finite moments of all orders. We can construct a sequence of polynomials $p_{0}(z), p_{1}(z), \ldots$ which form an orthonormal system in the Hilbert space $L^2(m)$ and such that the degree of $p_{k}$ is $k$. At least for some special classes of measures $m$ the empirical distribution of zeros $\frac 1k \mu_{p_k}$ of the deterministic polynomial $p_k$ converges, as $k\to\infty$, to some limiting measure which can be interpreted as the equilibrium measure (for the logarithmic potential without external field) on the support of $m$. The simplest example are Chebyshev (or, more generally, Jacobi) polynomials whose zeros have limiting arcsine distribution.
However, the equilibrium property of the limiting empirical distribution may fail even for very simple measures $m$. For example, if $m$ is rotationally invariant, then $p_k$ is a multiple of $z^k$ having a unique zero of multiplicity $k$ at $0$.

We are going to show how the potential theoretic interpretation of the limiting distribution of zeros can be restored if we pass to \textit{random} polynomials. Let $m_1,m_2,\ldots$ be finite, rotationally invariant measures on $\C$ having moments of all orders.  The polynomials which are orthonormal with respect to $m_n$ have the form $p_{k,n}(z)=f_{k,n}z^k$, $k\in\N_0$, where
\begin{equation}\label{eq:def_fkn_orthogonal}
f_{k,n}
=
\left(\int_{\C}|z|^{2k}m_n(dz)\right)^{-1/2}.
\end{equation}
Consider a random linear combination of these polynomials
$$
\GGG_n^{(\kappa)}(z)=\sum_{k=0}^{[\kappa n]} \xi_k f_{k,n}z^k.
$$
Some of the ensembles of random polynomials we considered above are particular cases of this setting. For example, taking $m_n(dz)=\frac{n}{\pi}e^{-n|z|^2}\lambda(dz)$ we have $p_{k,n}=(z \sqrt n)^{k}/\sqrt{k!}$  thus recovering the Weyl polynomials. Kac polynomials are recovered if we take $m_n=\lambda_{\TTT}$, the uniform probability distribution on $\TTT$. See Table~\ref{tab:orthogonal_poly} for more examples.  Observe that in all examples listed in Table~\ref{tab:orthogonal_poly}  the sequence $m_n$ satisfies a large deviation principle (see~\cite{dembo_zeitouni_book}) with speed $n$ and a certain rotationally invariant rate function $2Q(|z|)$.  This means that
$$
\limsup_{n\to\infty} \frac 1n \log m_n(K)\leq -2Q(K),
\;\;\;
\liminf_{n\to\infty} \frac 1n \log m_n(U)\geq -2Q(U)
$$
for every closed set $K\subset \C$ and every open set $U\subset \C$. Here, $Q(A)=\inf_{z\in A} Q(|z|)$ for any set $A\subset \C$.
For example, in the case of Weyl polynomials we have $Q(|z|)=\frac 12 |z|^2$. On the other hand, in the case of Weyl polynomials, the external field in Proposition~\ref{prop:equilibrium} is  given by $I(\log |z|)=\frac 1 2 |z|^2$.  This coincidence is a particular case of a more general statement, as we will argue now. We give the main idea skipping technical details.
By  Varadhan's lemma applied to~\eqref{eq:def_fkn_orthogonal} we have for all $t\geq 0$,
\begin{equation}\label{eq:varadhan_lemma}
\lim_{n\to\infty} \frac 1n \log f_{tn,n} =  \sup_{r>0} (t\log r - Q(r)) = \sup_{s\in\R} (ts - Q(e^s)) =: u(t).
\end{equation}
Imposing additional minor assumptions it is possible to check that we can apply Theorem~\ref{theo:general} and Proposition~\ref{prop:equilibrium} to $\GGG_n^{(\kappa)}$ with $I$ being the Legendre--Fenchel dual of $u$. However, the Legendre--Fenchel dual of $u$ is $Q(e^s)$ by~\eqref{eq:varadhan_lemma}.  We have thus identified the external field $I(\log |z|)$ with the rate function $Q(|z|)$. Some special cases are listed in Table~\ref{tab:orthogonal_poly}.

In the case of a not necessarily rotationally invariant weight $m_n=m$ which does not depend on $n$, the zeros of random combinations of orthogonal polynomials (with Gaussian coefficients) have been studied in~\cite{shiffman_zelditch1,bloom_shiffman}. Under appropriate conditions on $m$ it has been shown  that the zeros of such polynomials are asymptotically distributed according to the equilibrium measure on the support $S$ of $m$. The external potential in this case is equal to $+\infty$ outside  $S$  and is $0$ on $S$. The external potential  just restricts the equilibrium  measure to $S$. This setting includes the Kac polynomials but it does not include other examples considered here.

{\tiny
\begin{table}[t]
\centering
\caption{Random analytic functions associated to orthogonal polynomials with rotationally invariant weights}
\begin{tabular}{|l|ccccc|}
\hline
\rule{0pt}{5mm} \rule[-4mm]{0pt}{0pt}Model& $p_{k,n}$ & $m_n$ & $Q(|z|)$ &$\mu^{(\kappa)}$ &$S^{(\kappa)}$\\
\hline
\rule{0pt}{5mm} \rule[-4mm]{0pt}{0pt}Kac&  $z^k$ & $\lambda_{\TTT}$ &
$\begin{cases}0,& z \in \TTT,\\ +\infty, &z\notin \TTT\end{cases}$
&$\kappa \lambda_{\TTT}$& $\TTT$\\
\hline
\rule{0pt}{5mm} \rule[-4mm]{0pt}{0pt}Flat (Weyl)&  $\frac{(z\sqrt n)^k}{\sqrt {k!}}$ & $\frac{n}{\pi}e^{-n|z|^2}$ &
$\frac 12 |z|^2$
&$\frac 1 {\pi} \lambda$& $\bar \DDD_{\sqrt {\kappa}}$ \\
\hline
\rule{0pt}{5mm} \rule[-4mm]{0pt}{0pt}Elliptic&  $\sqrt{\frac{n(n-1)\ldots (n-k+1)}{k!}}z^k$ & $\frac{1}{\pi} \frac{n+1}{(1+|z|^2)^{n+2}}$ &
$\frac 12 \log (1+|z|^2)$
&$\frac 1 {\pi} \frac{1}{(1+|z|^2)^2}$& $\bar \DDD_{\sqrt {\frac{\kappa}{1-\kappa}}}$ \\
\hline
\rule{0pt}{5mm} \rule[-4mm]{0pt}{0pt}
\begin{tabular}[x]{@{}l@{}}Hyperbolic\\($z\in\DDD$)\end{tabular} &  $\sqrt{\frac{n(n+1)\ldots (n+k-1)}{k!}}z^k$ & $\frac{n-1}{\pi} (1-|z|^2)^{n-2}$ &
$-\frac 12 \log (1-|z|^2)$
&$\frac 1 {\pi} \frac{1}{(1-|z|^2)^2}$& $\bar \DDD_{\sqrt {\frac{\kappa}{1+\kappa}}}$ \\
\hline
\rule{0pt}{5mm} \rule[-4mm]{0pt}{0pt}
\begin{tabular}[x]{@{}l@{}}Theta\\($z\in\C\backslash \DDD$)\end{tabular}&  $e^{-\frac{(k+1)^2}{2n}}z^k$ &
$\pi^{-\frac 32}\sqrt n |z|^{-n\log |z|}$
&
$(\log |z|)^2$
&
$\frac 1 {|z|^2}$& $\bar \DDD_{e^{\kappa}}\backslash \DDD$ \\
\hline
\end{tabular}
\label{tab:orthogonal_poly}
\end{table}
}
{



\subsection{Open questions}\label{subsec:open_questions}
We established the global universality  for the distribution of complex zeros under a logarithmic moment condition. We strongly believe that \textit{local} universality for \textit{complex} zeros, as well as \textit{local} and \textit{global} universality for \textit{real} zeros for the models considered here can be proved under stronger moment conditions. For example, there should be local universality for distributions from the domain of attraction of an $\alpha$-stable law, $\alpha\in (0,2]$, however, different values of $\alpha$ should correspond to different local universality classes. The distribution of zeros of the Kac polynomials $\KKK_n$ in the case of logarithmically decaying tails has been studied in~\cite{kabluchko_zaporozhets}. We believe that in the general setting considered here, similar results should hold. In particular, the zeros should have a circle structure similar to that found in~\cite{kabluchko_zaporozhets}; see Figure~\ref{fig:weyl_universality} (right). It should be possible to generalize our results to (systems of) random analytic functions in many variables.
In Theorems~\ref{theo:LO_poly}, \ref{theo:LO_entire} we proved the a.s.\ convergence for the empirical distribution of zeros in Littlewood--Offord models. We used the natural nested structure of these models. It is open whether convergence in probability can be replaced by the a.s.\ convergence in Theorem~\ref{theo:invariant}. We don't know whether the approach of Section~\ref{subsec:log_potential} can be extended to polynomials which are orthogonal with respect to not necessarily rotationally invariant weights. 

\section{Proofs: Special cases}\label{sec:proofs_special_cases}
We are going to prove the results of Section~\ref{sec:intro}. We will verify  the assumptions of Section~\ref{subsec:general_theo} and apply Theorem~\ref{theo:general}. Recall the notation $u(t)=-\log f(t)$.
\begin{proof}[Proof of Theorem~\ref{theo:kac_generalized}.]
We can assume that $w=0$ since otherwise we can consider the polynomial $\WWW_n(e^{-w}z)$. It follows from $\lim_{k\to\infty}\frac 1k \log |w_k|=0$ that assumptions~(A1)--(A4) of Section~\ref{subsec:general_theo} are fulfilled with $T_0=1$, $R_0=+\infty$ and
$$
f(t)=
\begin{cases}
1,  &t\in[0,1],\\
0, &t>1,
\end{cases}
\;\;\;
u(t)=
\begin{cases}
0,  &t\in[0,1],\\
+\infty, &t>1.
\end{cases}
$$
The Legendre--Fenchel transform of $u$ is given by $I(r)=\max (0, r)$.
It follows from~\eqref{eq:theo:general} that $\mu$ is the uniform probability measure on $\TTT$.
\end{proof}
\begin{remark}
Under a slightly more restrictive assumption $\E \log |\xi_0|<\infty$, Theorem~\ref{theo:kac_generalized} can be deduced from the result of~\citet{hughes_nikeghbali} (which is partially based on the Erd\H{o}s--Turan inequality). This method, however, requires a subexponential growth of the coefficients and therefore fails in all other examples we consider here.
\end{remark}

\begin{proof}[Proof of Theorem~\ref{theo:invariant}]
By the Stirling formula, $\log n!=n\log n-n+o(n)$ as $n\to\infty$.
It follows that assumption~(A3) holds with
$$
u(t)=
\begin{cases}
\alpha(t\log t+ (1-t)\log (1-t)),  &(0\leq t\leq 1, \text{ elliptic case}),\\
\alpha(t\log t-t), &(t\geq 0, \text{ flat case)},\\
\alpha(t\log t- (1+t)\log (1+t)), &(t\geq 0, \text{ hyperbolic case)}.
\end{cases}
$$
In the elliptic case, $u(t)=+\infty$ for $t>1$. The Legendre--Fenchel transform of $u$ is given by
$$
I(s)=
\begin{cases}
\alpha \log (1+e^{s/\alpha}), &(s\in\R, \text{ elliptic case}),\\
\alpha e^{s/\alpha}, &(s\in \R,\text{ flat case}),\\
-\alpha \log(1-e^{s/\alpha}), &(s<0, \text{ hyperbolic case}).
\end{cases}
$$
In the hyperbolic case,  $I(s)=+\infty$ for $s\geq 0$. We have $R_0=1$ in the hyperbolic case and $R_0=+\infty$ in the remaining two cases. The proof is completed by applying Theorem~\ref{theo:general}.
\end{proof}

\begin{proof}[Proof of Theorem~\ref{theo:LO_poly}]
We are going to apply Theorem~\ref{theo:general} to the polynomial $\GGG_n(z)=\WWW_n(e^{\beta} n^{\alpha} z)$.  We have $f_{k,n}=e^{\beta k + \alpha k \log n} w_{k}$. Equation~\eqref{eq:fk_asympt}  implies that assumption~(A3) is satisfied with
$$
u(t)
=
\begin{cases}
\alpha(t\log t-t), &t\in [0,1],\\
+\infty, &t>1.
\end{cases}
$$
The Legendre--Fenchel transform of $u$ is given by
$$
I(s)=
\begin{cases}
\alpha e^{s/\alpha}, &s\leq 0,\\
\alpha + s, &s\geq 0.
\end{cases}
$$
Applying Theorem~\ref{theo:general} we obtain that $\frac 1n \mu_{\GGG_n}$ converges in probability to the required limit. A.s.\ convergence will be demonstrated in Section~\ref{subsec:proof_LO_poly_as} below.
\end{proof}

\begin{proof}[Proof of Theorem~\ref{theo:LO_entire}]
We apply Theorem~\ref{theo:general} to $\GGG_n(z)=\WWW(e^{\beta} n^{\alpha} z)$.
We have $u(t)=\alpha(t\log t-t)$ for all $t\geq 0$. Hence, $I(s)=\alpha e^{s/\alpha}$ for all $s\in\R$. We can apply Theorem~\ref{theo:general} to prove convergence in probability. A.s.\ convergence will be demonstrated in Section~\ref{subsec:proof_LO_entire_as} below.
\end{proof}

\begin{proof}[Proof of Theorem~\ref{theo:stable}]
Put $\sigma=+1$ in the case $\alpha>1$ and $\sigma=-1$ in the case $\alpha<1$.
We have $u(t)=\sigma t^{\alpha}$ for $t\geq 0$. It follows that
$$
I(r)
=
\begin{cases}
\sigma (\alpha-1) \left(\frac {\sigma r}{\alpha}\right)^{\frac {\alpha}{\alpha-1}}, &\sigma r\geq 0,\\
+\infty, &\sigma r<0.
\end{cases}
$$
We can apply Theorem~\ref{theo:general}.
\end{proof}

\section{Proof of Theorem~\ref{theo:log_asympt}}\label{sec:proof_log_asympt}
\subsection{The logarithmic moment condition}
The next lemma is the only place in our proofs where the logarithmic moment condition appears.
\begin{lemma}\label{lem:log_moment}
Let $\xi_0,\xi_1,\ldots$ be i.i.d.\ random variables. Fix $\eps>0$. Then,
\begin{equation}\label{eq:def_max_xi_k}
M:=\sup_{k=0,1,\ldots} \frac{|\xi_k|}{e^{\eps k}} < +\infty \text{ a.s.}
\;\;\;\Longleftrightarrow\;\;\;
\E \log (1+|\xi_0|)<\infty.
\end{equation}
\end{lemma}
\begin{proof}
For every non-negative random variable $X$ we have
$$
\sum_{k=1}^{\infty} \P[X\geq k] \leq  \E X\leq \sum_{k=0}^{\infty}\P[X\geq k].
$$
With $X=\frac 1 {\eps}\log (1+|\xi_0|)$ it follows that $\E \log(1+|\xi_0|)<\infty$ if and only if
$
\sum_{k=1}^{\infty}\P[|\xi_0|\geq e^{\eps k}-1] < \infty 
$ for some (equivalently, every) $\eps>0$.
The proof is completed by applying the Borel--Cantelli lemma.
\end{proof}

\subsection{Upper bound}
Take some $z\in\DDD_{R_0}\backslash\{0\}$. Fix an $\eps>0$.  We will show that
\begin{equation}\label{eq:upper_bound}
\limn \P\left[|\GGG_n(z)|>e^{n(I(\log |z|)+4\eps)}\right]=0.
\end{equation}
First, we estimate the tail of the Taylor series~\eqref{eq:def_GGG} defining $\GGG_n$. By making $\eps$ smaller we may assume that  $|z|e^{2\eps}<R_0$. By assumption~(A4) there is $A>0$ such that for all $n\geq A$ and all $k\geq An$,
$$
|f_{k,n}|< (|z| e^{2\eps})^{-k}.
$$
Lemma~\ref{lem:log_moment} implies that for some a.s.\ finite random variable $M$,
\begin{equation}\label{eq:upper_bound_1}
\left|\sum_{k\geq An}\xi_k f_{k,n} z^k\right|
\leq
M \sum_{k\geq An} e^{\eps k} |f_{k,n}| |z|^k
\leq
M \sum_{k\geq An} e^{-\eps k}
\leq
M.
\end{equation}
The last inequality holds if $n$ is sufficiently large. Note in passing that this implies that for large $n$ the series~\eqref{eq:def_GGG} converges with probability $1$.

We now consider the beginning of the Taylor series~\eqref{eq:def_GGG} defining $\GGG_n$. Take some $\delta>0$. By assumption~(A3), there is $N$ such that for all $n>N$ and all $k\leq An$,
\begin{equation}\label{eq:est_f_kn}
|f_{k,n}|<\left(f(k/n)+\delta\right)^n.
\end{equation}
It follows from~\eqref{eq:def_I} that for all $t\geq 0$,
\begin{equation}\label{eq:est_I_log_z}
t\log |z| + \log f(t) \leq I(\log |z|).
\end{equation}
Using~\eqref{eq:est_f_kn}, \eqref{eq:est_I_log_z} and Lemma~\ref{lem:log_moment} with $\eps/A$ instead of $\eps$ we obtain that there is an a.s.\ finite random variable $M'$ such that
\begin{align}
\left|\sum_{0\leq k< An}\xi_k f_{k,n} z^k\right|
&\leq
M' \sum_{0\leq k< An} e^{\frac{\eps k}{A}} \left(f\left(\frac kn\right)+\delta\right)^n  |z|^k \label{eq:upper_bound_2}\\
&\leq
M' e^{\eps n} \sum_{0\leq k< An} \left(e^{\frac kn\log |z|+\log f(\frac kn)}+ \delta |z|^{\frac k n}\right)^n \notag \\
&\leq
M' e^{2 \eps n} \left(e^{I(\log |z|)}+\delta |z|^A\right)^n \notag \\
&\leq
M' e^{3\eps n} e^{n I(\log |z|)}, \notag
\end{align}
where the last inequality holds if $\delta=\delta(\eps)$ is sufficiently small.
Combining~\eqref{eq:upper_bound_1} and~\eqref{eq:upper_bound_2}, we obtain that for large $n$,
\begin{equation}\label{eq:upper_bound_G}
|\GGG_n(z)|\leq M'  e^{n(I(\log |z|)+3\eps)} + M.
\end{equation}
Since $M$ and $M'$ are a.s.\ finite  by Lemma~\ref{lem:log_moment}, this implies~\eqref{eq:upper_bound}.

\subsection{Lower bound}
Fix $\eps>0$ and $z\in\DDD_{R_0}\backslash\{0\}$.  We show that
\begin{equation}\label{eq:proof_log_part_func_2}
\P\left[|\GGG_n(z)|< e^{n(I(\log |z|)-4\eps)}\right]=O\left(\frac 1 {\sqrt n}\right),
\;\;\;
n\to\infty.
\end{equation}

We will use the Kolmogorov--Rogozin inequality in a multidimensional form which can be found in~\cite{esseen}. Given a $d$-dimensional random vector $X$ define its concentration function by
\begin{equation}\label{eq:def_concentration_func}
Q(X; r) = \sup_{x\in \R^d} \P[X\in \DDD_r(x)], \qquad r>0,
\end{equation}
where $\DDD_r(x)$ is a $d$-dimensional ball of radius $r$ centered at $x$. An easy consequence of~\eqref{eq:def_concentration_func} is that for all independent random vectors $X,Y$ and all $r,a>0$,
\begin{equation}\label{eq:prop_concentration_func}
Q(X+Y; r)\leq Q(X; r),
\qquad
Q(aX;r)=Q(X; r/a).
\end{equation}
 The next result follows from Corollary~1 on p.~304 of~\cite{esseen}.
\begin{theorem}[Kolmogorov--Rogozin inequality]
\label{theo:kolmogorov-rogozin}
There is a constant $C_d$ depending only on $d$ such that for all independent (not necessarily identically distributed) random $d$-dimensional vectors $X_1,\ldots,X_n$ and for all $r>0$, we have
$$
Q(X_1+\ldots+X_n; r) \leq C_d \cdot \left( \sum_{k=1}^n (1-Q(X_k; r))\right)^{-1/2}.
$$
\end{theorem}
The idea of our proof of~\eqref{eq:proof_log_part_func_2} is to use the Kolmogorov--Rogozin inequality to show that the probability of very strong cancellation among the terms of the series~\eqref{eq:def_GGG} defining $\GGG_n$ is small. First, we have to single out those terms of $\GGG_n$ in which the coefficient $f_{k,n}$ is large enough.
By definition of $I$, see~\eqref{eq:def_I}, there is $t_0\in[0, T_0]$  such that $t_0\log |z|+\log f(t_0)>I(\log |z|)-\eps$. Moreover, by assumptions~(A2), we can find a closed interval $J$ of length $|J|>0$ containing $t_0$ such that
$$
f(t) |z|^t >e^{I(\log |z|)-2\eps},\;\;\; t\in J.
$$
Define a set $\JJJ_n=\{k\in\N_0: k/n \in J\}$. By assumption~(A3) there is $N$ such that for all $n>N$ and all $k\in \JJJ_n$,
$$
|f_{k,n}| |z|^k > e^{n(I(\log |z|)-3\eps)}.
$$
 For $k=0,\ldots,n$ define $$
a_{k,n}=e^{-n(I(\log |z|)-3\eps)} f_{k,n} z^k.
$$
Note that $|a_{k,n}| > 1$ for $k\in \JJJ_n$.
Define
$$
\GGG_{n,1}
=
\sum_{k\in \JJJ_n} a_{k,n}\xi_k,
\qquad
\GGG_{n,2}
=
\sum_{k\notin \JJJ_n} a_{k,n}\xi_k.
$$
By taking the real and imaginary parts we can view the complex random variables $a_{k,n}\xi_k$ as two-dimensional random vectors.
Using~\eqref{eq:prop_concentration_func} we arrive at
\begin{equation}\label{eq:proof_log_part_func_2a}
\P[|\GGG_n(z)|< e^{n(I(\log |z|)-4\eps)}]
\leq
Q(\GGG_{n,1}+\GGG_{n,2}; e^{-\eps n})
\leq
Q(\GGG_{n,1}; e^{-\eps n}).
\end{equation}
By Theorem~\ref{theo:kolmogorov-rogozin}, there is an absolute constant $C$ such that for all $r>0$,
$$
Q(\GGG_{n,1}; r)
\leq
C\cdot \left(\sum_{k\in \JJJ_n}(1-Q(a_{k,n}\xi_k ; r)) \right)^{-1/2}
\leq
C\cdot \left(\sum_{k\in \JJJ_n}(1-Q(\xi_k ; r)) \right)^{-1/2}.
$$
Here, the second inequality follows from the fact that $|a_{k,n}| > 1$. Now, since the random variable $\xi_0$ is supposed to be non-degenerate, we can choose $r>0$ so small that $Q(\xi_0; r)<1$. Note that this is the only place in the proof of Theorem~\ref{theo:general} where we use randomness. The rest of the proof is valid for any deterministic sequence $\xi_0,\xi_1,\ldots$ such that $|\xi_n|=O(e^{\delta n})$ for every $\delta>0$.  If $n$ is sufficiently large, then $e^{-\eps n}\leq r$ and hence,
\begin{equation}\label{eq:proof_log_part_func_2b}
Q(\GGG_{n,1}; e^{-\eps n})\leq Q(\GGG_{n,1}; r)
\leq C_1  |\JJJ_n|^{-1/2}
\leq C_2 n^{-1/2}.
\end{equation}
In the last inequality we have used that the number of elements of $\JJJ_n$ is larger than $(|J|/2) n$ for large $n$. Taking~\eqref{eq:proof_log_part_func_2a} and~\eqref{eq:proof_log_part_func_2b} together completes the proof of the lower bound~\eqref{eq:proof_log_part_func_2}.

\section{Proof of Theorem~\ref{theo:general} and related results}\label{sec:proof_general}
\subsection{Proof of Theorem~\ref{theo:general}}
Recall that $\mu_{\GGG_n}$ is the measure counting the zeros of $\GGG_n$. Our aim is to show that for every smooth, compactly supported function $\varphi: \DDD_{R_0} \to\R$,
\begin{equation}\label{eq:proof_theo_general}
\frac 1n \int_{\DDD_{R_0}} \varphi(z)\mu_{\GGG_n}(dz)
\toprobab
\int_{\DDD_{R_0}} \varphi(z) \mu(dz).
\end{equation}
Let $E_n$ be the event $\GGG_n\equiv 0$. The left-hand side of~\eqref{eq:proof_theo_general} is not well-defined on $E_n$. However, we will argue that $\lim_{n\to\infty} \P[E_n]=0$. By assumptions~(A1), (A2) and~(A3), if $n$ is large, then $f_{k,n}>0$ for all $k<T_0n/2$. Also, $\P[\xi_0=0]<1$. It follows that $\lim_{n\to\infty} \P[E_n]=0$. In our proof of~\eqref{eq:proof_theo_general} we may restrict ourselves to the complement of $E_n$.

It is known, see~\cite{peres_book}, that  on the complement of $E_n$,
\begin{equation}\label{eq:zeros_laplacian1}
\int_{\DDD_{R_0}} \varphi(z)\mu_{\GGG_n}(dz)
=
\frac 1 {2\pi} \int_{\DDD_{R_0}} \Delta \varphi(z)  \log |\GGG_n(z)| \lambda(dz).
\end{equation}
This is just a restatement of~\eqref{eq:mu_GGG_is_Laplace_log_GGG}. From Theorem~\ref{theo:log_asympt} we know that for every $z\in\DDD_{\R_0}\backslash\{0\}$, the random variable $p_n(z):=\frac 1n \log |\GGG_n(z)|$ converges to $p(z):=I(\log |z|)$ in probability.  Assuming for a moment that we can pass to the limit under the sign of integral (which will be justified later), we have
\begin{equation}\label{eq:Sn_varphi_converges}
\frac 1n \int_{\DDD_{R_0}} \varphi(z)\mu_{\GGG_n}(dz)
\toprobab
\frac 1 {2\pi} \int_{\DDD_{R_0}} \Delta \varphi(z) p(z) \lambda(dz).
\end{equation}
This proves Theorem~\ref{theo:general}, but with the formula  $\frac 1 {2\pi}\Delta (I(\log |z|))$ for the limiting distribution of zeros, as in Remark~\ref{rem:mu_laplace_I_log}. We prove now that the limiting distribution of zeros can be given by~\eqref{eq:theo:general}. We claim that with $\mu$ defined by~\eqref{eq:theo:general},
\begin{equation}\label{eq:laplace_selfadjoint}
\frac 1 {2\pi} \int_{\DDD_{R_0}} \Delta \varphi(z) p(z) \lambda(dz)=\int_{\DDD_{R_0}} \varphi(z)\mu(dz).
\end{equation}
If $I$ is smooth, then by Green's identity, the left-hand side of~\eqref{eq:laplace_selfadjoint} is equal to
$$
\frac 1 {2\pi}
\int_{\DDD_{R_0}} \varphi(z) \Delta (I(\log |z|)) \lambda(dz)
=
\int_{\DDD_{R_0}} \varphi(z) \frac{I''(\log |z|)}{2\pi |z|^2} \lambda(dz)
=
\int_{\DDD_{R_0}} \varphi(z)\mu(dz).
$$
If $I$ is not smooth, we can find a sequence of non-decreasing, smooth functions $I'_1,I'_2,\ldots$ such that $I_n'\leq I'$ and $\lim_{n\to\infty}I_n'(s)=I'(s)$ for all $s<\log R_0$ where $I'$ is continuous. By dominated convergence, $I_n(s):=\int_{-\infty}^s I_n'(t)dt \to I(s)$ as $n\to\infty$. For each $I_n$ we can use the Green's identity as above, and then let $n\to\infty$ to obtain~\eqref{eq:laplace_selfadjoint} in full generality.

\subsection{Dominated convergence}
It remains to justify the interchanging of the limit and the integral when passing from~\eqref{eq:zeros_laplacian1} to~\eqref{eq:Sn_varphi_converges}. Recall that a sequence of random variables $X_n$ is bounded in probability (or tight) if for every $\eps>0$ we can find $A=A(\eps)$ such that $\P[|X_n|>A]<\eps$ for all $n\in\N$. Recall also that $X_n$ is called bounded a.s.\ if $\limsup_{n\to\infty}|X_n|<\infty$ a.s.
We need a lemma  from~\cite{tao_vu}.
\begin{lemma}[Lemma~3.1 in~\cite{tao_vu}]\label{lem:tao_vu}
Let $(X, \mathcal A, \nu)$ be a finite measure space. Let $f_n:X\to\R$  be random functions defined on a probability space $(\Omega, \mathcal B, \P)$ which are jointly measurable with respect to $\mathcal A\otimes \mathcal B$. Assume that for $\nu$-a.e.\ $x\in X$ we have $f_n(x)\to 0$ in probability (resp., a.s.) and that the sequence $\int_X |f_n(x)|^{1+\delta} \nu(dx)$ is bounded in probability (resp., a.s.) for some $\delta>0$. Then, $\int_X f_n(x)\nu(dx)$ converges in probability (resp., a.s.) to $0$.
\end{lemma}

Recall that the function $\varphi$ vanishes outside some disk  $\DDD_r$, where $r<R_0$. By Lemma~\ref{lem:tao_vu}, the passage from~\eqref{eq:zeros_laplacian1} to~\eqref{eq:Sn_varphi_converges} is justified if we show that the sequence of random variables
\begin{equation}\label{eq:def_Bn}
B_n:= \int_{\DDD_{r}} |\Delta \varphi(z)|^2 |p_n(z)-p(z)|^2 \lambda(dz)
\end{equation}
is tight.  Since $\varphi$ is bounded on $\DDD_r$, we have
$$
B_n\leq C \int_{\DDD_r} |p(z)|^2 \lambda(dz) + C\int_{\DDD_r} |p_n(z)|^2 \lambda(dz).
$$
The first summand on the right-hand side is a finite constant since $0\leq I(\log |z|)\leq I(\log r).$
To complete the proof of Theorem~\ref{theo:general}  we need to  demonstrate the tightness of the sequence
\begin{equation}\label{eq:def_tilde_Bn}
\tilde B_n:=\frac 1 {n^2} \int_{\DDD_r} (\log |\GGG_n(z)|)^2 \lambda(dz).
\end{equation}

Let us consider first the case in which $\GGG_n$ is a polynomial of degree $n$ and $f(1)\neq 0$. This applies for example in the setting of Theorem~\ref{theo:LO_poly}. Denoting by $w_{1n},\ldots,w_{nn}$ the zeros of $\GGG_n$ we can write
$$
\GGG_n(z) = \xi_n f_{n,n}  (z-w_{1n})\ldots (z-w_{nn}).
$$
Using the inequality of the  arithmetic and quadratic means two times we obtain
$$
\tilde B_n\leq \frac {2\pi r^2}{n^2}(\log |f_{n,n}| + \log |\xi_n|)^2 + \frac 2 n \sum_{k=1}^n \int_{\DDD_r} \log^2 |z-w_{kn}| \lambda(dz).
$$
Recall that $\lim_{n\to\infty}\frac 1n \log |f_{n,n}|=\log f(1)$ is finite. Hence, the first term on the right-hand side is bounded in probability.
The second term is bounded by a deterministic constant since $\int_{\DDD_r} \log^2 |z-w| \lambda(dz)\leq C$ for some constant $C=C(r)$ not depending on $w\in\C$.

In the general setting the proof of tightness of $\tilde B_n$ is more involved. The main difficulty is that $\log |\GGG_n|$ becomes infinite at zeros of $\GGG_n$. Thus, we have to show that with high probability, $\GGG_n$ has not too many zeros. We will prove that $\tilde B_n$ is bounded a.s.:
\begin{equation}\label{eq:tao_vu_cond_as}
\limsup_{n\to\infty} \tilde B_n < \infty \text{ a.s.}
\end{equation}

If $\GGG_n(0)=0$ we let $\tau_n$ be the multiplicity of the zero at $0$.
Write $\GGG_n^*(z)=\GGG_n(z)/z^{\tau_n}$ and define $\GGG_n^*(0)\neq 0$ by continuity.
Let $M_n(R)=\sup_{|z|=R} |\GGG_n^*(z)|$, where $R<R_0$. First of all, it follows from~\eqref{eq:upper_bound_G} that
\begin{equation}\label{eq:M_n_upper_bound}
D_1(R):=\limsup_{n\to\infty} \frac 1n \log M_n(R) < \infty \text{ a.s.}.
\end{equation}
Note that $M'$ and $M$ there do not depend on $\arg z$. Let $N_n(R)$ be the number of zeros of $\GGG_n^*$ in the disk $\DDD_R$, counting multiplicities.   Denote by $a_{1,n},\ldots, a_{N_n(R), n}$ the zeros of $\GGG_n^*$ in the disk $\DDD_R$. The Poisson--Jensen formula, see, e.g., \cite[Chapter~8]{markushevich_book}, states that for every $z\in \DDD_r$ and every $R>r$ we have
\begin{equation}\label{eq:poisson_jensen}
\log \left|\GGG_n^*(z)\right|
=
I_n(z;R)
+
\sum_{k=1}^{N_n(R)} \log \left|\frac{R(z-a_{k,n})}{R^2-\bar a_{k, n} z} \right|,
\end{equation}
where
\begin{equation}\label{eq:poisson_integral}
I_n(z;R)=\frac 1 {2\pi} \int_{0}^{2\pi} \log |\GGG_n^*(R e^{i\theta})|P_R(|z|, \theta-\arg z)d\theta
\end{equation}
and $P_R$ is the Poisson kernel:
$$
P_R(\rho, \varphi)=\frac{R^2-\rho^2}{R^2+\rho^2-2R\rho \cos \varphi}.
$$
Fix a small $\delta>0$. Note that $0\leq P_{r+\delta}(\rho, \varphi)<C$ for all $\rho<r$ and $\varphi\in[0,2\pi]$, where $C$ depends only on $r$ and $\delta$. It follows that
$
I_n(z; r+\delta)\leq C \log M_n(r+\delta)
$
for all $z \in\DDD_r$ and hence, by~\eqref{eq:M_n_upper_bound},
\begin{equation}\label{eq:I_n_upper_bound}
D_2:=\limsup_{n\to\infty} \frac 1n \sup_{z\in\DDD_r} I_n(z;r+\delta) < \infty \text{ a.s.}
\end{equation}
We will show that
\begin{equation}\label{eq:probab_too_many_zeros}
\limsup_{n\to\infty} \frac 1n N_n(r+\delta) < \infty \text{ a.s.}
\end{equation}
It follows from the Poisson--Jensen formula~\eqref{eq:poisson_jensen} with $R=r+2\delta$ and $z=0$ that
\begin{equation}\label{eq:poisson_jensen_number_zeros}
N_n(r+\delta) \leq \frac 1 {\log \frac{r+2\delta}{r+\delta}} \log \left|\frac{M_n(r+2\delta)}{ \GGG_n^*(0)}\right|.
\end{equation}
Indeed, any zero of $\GGG_n^*$ in $\DDD_{r+\delta}$ gives contribution at most $\log \frac{r+\delta}{r+2\delta}$ to the sum on the right-hand side of~\eqref{eq:poisson_jensen} and $I_n(0;R)\leq M_n(R)$. Clearly, $\GGG_n^*(0)=f_{\tau_n,n}\xi_{\tau_n}$. Let $\tau=\min\{n: \xi_n\neq 0\}$ be the index of the first non-zero $\xi_n$.
Note that $\tau<\infty$ a.s.
By~(A1), (A2), (A3), for sufficiently large $n$ we have $f_{k,n}>0$ for all $k\in [0, T_0 n/2]$.  It follows that $\tau_n=\tau$ for sufficiently large $n$.  Hence,
\begin{equation}\label{eq:log_GGG_n_star_lower_bound}
\liminf_{n\to\infty}\frac 1n \log |\GGG_n^*(0)|=0 \text{ a.s.}
\end{equation}
Applying~\eqref{eq:M_n_upper_bound} and~\eqref{eq:log_GGG_n_star_lower_bound} to the right-hand side of~\eqref{eq:poisson_jensen_number_zeros} we arrive at~\eqref{eq:probab_too_many_zeros}.

Next we show that
\begin{equation}\label{eq:I_n_lower_bound}
D_3:=\liminf_{n\to\infty} \frac 1n \inf_{z\in\DDD_r} I_n(z;r+\delta) > -\infty \text{ a.s.}
\end{equation}
It is known that $I_n(0;R)=\frac 1 {2\pi} \int_0^{2\pi} \log |\GGG_n^*(R e^{i\theta})|d\theta$ is non-decreasing in $R$. Note that $I_n(0;0)=\log |\GGG_n^*(0)|$. Hence, by~\eqref{eq:log_GGG_n_star_lower_bound},
\begin{equation}\label{eq:I_n_zero_lower_bound}
\liminf_{n\to\infty} \frac 1n I_n(0;R) > -\infty \text{ a.s.}
\end{equation}
From now on we set $R=r+\delta$. Let $q_n(\theta)=\frac 1n \log |\GGG_n^*(Re^{i\theta})|$ and write
$$
q_n^+(\theta)= \max (q_n(\theta),0),\;\;\;
q_n^-(\theta)=-\min (q_n(\theta),0).
$$
Then, $q_n(\theta)=q_n^+(\theta)-q_n^-(\theta)$. Note that there is a constant $C>1$ depending only on $r,\delta$ such that $1/C<P_{R}(\rho, \varphi)<C$ for all $\rho<r$ and $\varphi\in [0,2\pi]$.
We have, for all $z\in\DDD_r$,
\begin{align*}
\frac {2\pi}n I_n(z;R)
&=
\int_0^{2\pi} q_n^+(\theta)P_R(|z|,\theta-\arg z)d\theta
- \int_0^{2\pi} q_n^-(\theta)P_R(|z|,\theta-\arg z)d\theta\\
&\geq
\frac 1 C\int_0^{2\pi} q_n^+(\theta)d\theta -
C\int_0^{2\pi} q_n^-(\theta)d\theta\\
&=
\frac {2\pi C} n I_n(0;R)  - \left(C-\frac 1 C\right)\int_0^{2\pi} q_n^+(\theta)d\theta\\
&\geq
\frac {2\pi C} n I_n(0;R)- \left(C-\frac 1 C\right)\frac{2\pi}{n} \log M_n(R).
\end{align*}
Recalling~\eqref{eq:M_n_upper_bound} and~\eqref{eq:I_n_zero_lower_bound} we arrive at~\eqref{eq:I_n_lower_bound}.

We are ready to prove that the sequence $\tilde B_n$ is bounded a.s.  Applying the inequality of the arithmetic and quadratic means to~\eqref{eq:poisson_jensen}, we obtain
$$
(\log |\GGG_n(z)|)^2
\leq
3(\tau_n \log |z|)^2 + 3I_n^2(z;R)+ 3N_n(R)\sum_{k=1}^{N_n(R)}  \log^2 \left|\frac{R(z-a_{k,n})}{R^2-\bar a_{k, n} z} \right|.
$$
There is a constant $C$ depending only on $r,\delta$ such that for all $a\in \DDD_R$,
$$
\int_{\DDD_r} \log^2 \left|\frac{R(z-a)}{R^2-\bar a z} \right|\lambda(dz)\leq C.
$$
Recalling~\eqref{eq:def_tilde_Bn} we have, for some constant $C$ depending only on $r,\delta$,
\begin{align*}
\tilde B_n
&\leq
\frac C{n^2}\left( \tau_n^2 + \sup_{z\in\DDD_r} I_n^2(z;R)  + N_n^2(R)\right).
\end{align*}
Recall that $\tau_n=\tau$ for sufficiently large $n$. Utilizing~\eqref{eq:I_n_upper_bound}, \eqref{eq:I_n_lower_bound}, \eqref{eq:probab_too_many_zeros}  we arrive at~\eqref{eq:tao_vu_cond_as}.
The sequence $\tilde B_n$ is bounded a.s.\ and hence, tight.

\subsection{Proof of the a.s.\ convergence in Theorem~\ref{theo:LO_poly}} \label{subsec:proof_LO_poly_as}
Convergence in probability has already been established in Section~\ref{sec:proofs_special_cases}. Given $n\in\N$ we can find a unique $j_n\in\N$ such that $j_n^3\leq n< (j_n+1)^3$. Write $m_n=j_n^3$ and $\GGG_n(z)=\WWW_n(e^{\beta} m_n^{\alpha}z)$. Note that $\lim_{n\to\infty}m_n/n =1$. We will show that $\frac 1n \mu_{\GGG_n}$ converges a.s.\ to the measure with density~\eqref{eq:LO_poly_density}. To this end, we need to prove the a.s.\ convergence of the corresponding potentials. Fix $z\in\DDD$. We will prove that
\begin{equation}\label{eq:LO_as_conv_potentials}
\frac 1n \log | \GGG_n(z)| \toas \alpha |z|^{1/\alpha}.
\end{equation}
Note that $\GGG_n$ satisfies all assumptions of Section~\ref{subsec:general_theo}.  It follows from~\eqref{eq:upper_bound_G} that
$$
\limsup_{n\to\infty} \frac 1n \log |\GGG_n(z)|\leq \alpha |z|^{1/\alpha} \text{ a.s.}
$$
Thus, we have to prove only the lower bound in~\eqref{eq:LO_as_conv_potentials}. Fix a small $\eps>0$. It follows from~\eqref{eq:proof_log_part_func_2} and the Borel--Cantelli lemma applied to the subsequence $\{j^3\}_{j\in\N}$ that  with probability $1$ for all but finitely many $n\in\N$,
\begin{equation}\label{eq:as_conv_lower_subseq}
|\GGG_{m_n}(z)| > e^{m_n(\alpha |z|^{1/\alpha}-\eps)}.
\end{equation}
Let now $n$ be a number not of the form $j^3$.  We have, by Lemma~\ref{lem:log_moment} and~\eqref{eq:fk_asympt},
\begin{align*}
|\GGG_n(z) - \GGG_{m_n}(z)|
&=
\left|\sum_{k=m_n+1}^{n} \xi_k w_k e^{\beta k} m_n^{\alpha k} z^k\right|\\
&\leq
M e^{2\eps n} \sum_{k=m_n+1}^{n} e^{-\alpha(k\log k-k)} n^{\alpha k} |z|^k.
\end{align*}
The function $x\mapsto -\alpha(x\log x-x)+\alpha x \log n$ defined for $x>0$ attains its maximum, which is equal to $\alpha n$, at $x=n$. Recall that $|z|<1$. Since $m_n>(1-\eps)n$ and $M<e^{\eps n}$ if $n$ is sufficiently large, we have the estimate
$$
|\GGG_n(z)-\GGG_{m_n}(z)| \leq e^{3\eps n} e^{\alpha n} |z|^{(1-\eps) n}.
$$
Since $\alpha+\log |z| < \alpha |z|^{1/\alpha}$, we have, if $\eps>0$ is small enough,
\begin{equation}\label{eq:as_conv_lower_diff}
|\GGG_n(z)-\GGG_{m_n}(z)| \leq e^{(1-\eps)n(\alpha |z|^{1/\alpha}-2\eps)}\leq e^{m_n (\alpha |z|^{1/\alpha}-2\eps)}.
\end{equation}
Bringing~\eqref{eq:as_conv_lower_subseq} and~\eqref{eq:as_conv_lower_diff} together we obtain that with probability $1$ for all but finitely many $n$,
$
|\GGG_n(z)|\geq  e^{m_n (\alpha |z|^{1/\alpha}-2\eps)}.
$
This is the required lower bound in~\eqref{eq:LO_as_conv_potentials}.

We are ready to complete the proof. We need to show that $\frac 1n \mu_{\GGG_n}$ converges a.s.\ to a measure $\mu$ with density~\eqref{eq:LO_poly_density}. Take any smooth function $\varphi:\C\to\R$ having a support contained in $\DDD$. Write $p_n(z)=\frac 1n \log |\GGG_n(z)|$. As in~\eqref{eq:zeros_laplacian1} we have
$$
S_n(\varphi)
:=\frac 1n \sum_{z\in\C: \GGG_n(z)=0}\varphi(z)
=\frac 1{2\pi} \int_{\DDD} \Delta \varphi(z) p_n(z)\lambda(dz).
$$
We have shown in~\eqref{eq:LO_as_conv_potentials} that for every $z\in\DDD$, $p_n(z)$ converges to $p(z)=\alpha |z|^{1/\alpha}$ a.s. Assuming that interchanging the limit and the integral is possible, we arrive at
\begin{equation}\label{eq:zeros_laplacian_as}
S_n(\varphi) \toas
\frac 1{2\pi} \int_{\C} \Delta \varphi(z) p(z)\lambda(dz)
=
\int_{\C} \varphi(z) \mu(dz).
\end{equation}
The last step follows from~\eqref{eq:laplace_selfadjoint}. To justify the interchanging of the limit and the integral we use Lemma~3.1 of~\cite{tao_vu}. To use the lemma, we need to show that $B_n$ defined in~\eqref{eq:def_Bn} (or, equivalently, $\tilde B_n$ defined by~\eqref{eq:def_tilde_Bn}) is bounded a.s. This means that $\limsup_{n\to\infty}\tilde B_n<\infty$ a.s. But we have already verified this in~\eqref{eq:tao_vu_cond_as}.

Unfortunately, we were able to establish~\eqref{eq:LO_as_conv_potentials} for $z\in\DDD$ only. That is why the support of $\varphi$ was restricted to $\DDD$. To complete the proof we have only to get rid of this assumption. Given a small $\eps>0$ let $\psi_{\eps}:\C\to [0,1]$ be a smooth function which is $1$ on $\DDD_{1-2\eps}$ and $0$ outside $\DDD_{1-\eps}$. Let $N_n(r)$ be the number of zeros of $\GGG_n$ inside $\DDD_{r}$. Then, by~\eqref{eq:zeros_laplacian_as},
\begin{equation}\label{eq:proof_as_outside_disk}
\frac 1n N_n(1-\eps)
\geq
S_n(\psi_{\eps})
\toas
\int_{\DDD} \psi_{\eps}(z)\mu(dz)
\geq
\mu(\DDD_{1-2\eps})
=
(1-2\eps)^{1/\alpha}.
\end{equation}
Let now $\varphi$ be an arbitrary smooth compactly supported function on $\C$. Write $\varphi=\varphi_{1,\eps}+\varphi_{2,\eps}$, where $\varphi_{1,\eps}=\varphi\cdot \psi_{\eps}$ is smooth with support in $\DDD_{1-\eps}$ and $\varphi_{2,\eps}=\varphi\cdot (1-\psi_{\eps})$ is smooth with support in $\C\backslash \DDD_{1-2\eps}$. Then, by~\eqref{eq:zeros_laplacian_as} and~\eqref{eq:proof_as_outside_disk},
$$
S_n(\varphi_{1,\eps})\toas \int_{\DDD} \varphi_{1,\eps}(z) \mu(dz), \;\;\;
\limsup_{n\to\infty} S_n(\varphi_{2,\eps}) \leq (1-(1-4\eps)^{1/\alpha}) \|\varphi\|_{\infty}.
$$
Note that $\varphi_{1,\eps}$ coincides with $\varphi$ on $\DDD_{1-2\eps}$ and hence, $\int_{\DDD} \varphi_{1,\eps}(z) \mu(dz)$ converges to $\int_{\DDD} \varphi(z) \mu(dz)$  as $\eps\downarrow 0$. Since $\eps>0$ was arbitrary, it follows that $S_n(\varphi)=S_n(\varphi_{1,\eps})+S_n(\varphi_{2,\eps})$ converges a.s.\ to $\int_{\DDD} \varphi(z)\mu(dz)$. This completes the proof.

\subsection{Proof of the a.s.\ convergence in Theorem~\ref{theo:LO_entire}}\label{subsec:proof_LO_entire_as}
Let $m_n$ be defined in the same way as in the previous proof. Write $\GGG_n(z)=\WWW(e^{\beta} m_n^{\alpha}z)$. We will show that the sequence of random measures $\frac 1n \mu_{\GGG_n}$ converges a.s.\ to the measure with density~\eqref{eq:LO_entire_density}. This implies Theorem~\ref{theo:LO_entire} since $\lim_{n\to\infty}m_n/n=1$. Fix $z\in\C$. We show that
\begin{equation}\label{eq:LO_as_conv_potentials_entire}
\frac 1n \log | \GGG_n(z)| \toas \alpha |z|^{1/\alpha}.
\end{equation}
Note that $\GGG_n$ satisfies the  assumptions of Section~\ref{subsec:general_theo} with $I(s)=\alpha e^{s/\alpha}$ there.   From~\eqref{eq:upper_bound_G} it follows that
$$
\limsup_{n\to\infty} \frac 1n \log |\GGG_n(z)|\leq \alpha |z|^{1/\alpha} \text{ a.s.}
$$
We prove the lower bound in~\eqref{eq:LO_as_conv_potentials_entire}. Fix small $\eps>0$. It follows from~\eqref{eq:proof_log_part_func_2} and the Borel--Cantelli lemma applied to the subsequence $\{j^3\}_{j\in\N}$ that  with probability $1$ for all but finitely many $n\in\N$,
$$
|\GGG_{m_n}(z)| > e^{m_n(\alpha |z|^{1/\alpha}-\eps)}.
$$
However, $\GGG_{m_n}(z)=\GGG_n(z)$ by definition. Also, $\lim_{n\to\infty} m_n/n=1$. This proves the lower bound in~\eqref{eq:LO_as_conv_potentials_entire}. The rest of the proof is the same as in Theorem~\ref{theo:LO_poly}, but we don't need to worry about the case $z\notin \DDD$.

\subsection{Proof of Theorem~\ref{theo:LO_poly_converse}}
Let $\WWW_n(z)=\sum_{k=0}^n \xi_k w_k z^k$, where $w_k$ is a sequence satisfying~\eqref{eq:fk_asympt} and~\eqref{eq:LO_additional_assumpt}.
Assume that $\E \log(1+|\xi_0|)=\infty$.  Fix $\eps>0$. We will show that with probability $1$ there exist infinitely many $n$ such that all zeros of $\WWW_n(e^{-\beta}n^{-\alpha} z)$ are located in the disk $\DDD_{\eps}$. We use an idea of~\cite{iz_log}.
By Lemma~\ref{lem:log_moment}, $\limsup_{n\to\infty} |\xi_n|^{1/n}=+\infty$. Hence, with probability $1$ there exist infinitely many $n$ such that
\begin{equation}\label{eq:no_log_moment_super_exp}
|\xi_n|^{\frac 1 n}>\max_{k=1,\ldots,n-1} |\xi_{n-k}|^{\frac 1{n-k}},
\;\;\;
|\xi_n|^{\frac 1n}>\max\left\{\frac{3C+1}{\eps}, \frac 1 {e^{\alpha}\eps}\right\}.
\end{equation}
Let $n$ be such that~\eqref{eq:no_log_moment_super_exp} holds. By~\eqref{eq:LO_additional_assumpt} and~\eqref{eq:no_log_moment_super_exp}, we have for every $z\in\C$ and $k<n$,
\begin{align*}
\left|w_{n-k}\xi_{n-k}\left(\frac z {e^{\beta} n^{\alpha}}\right)^{n-k}\right|
&\leq
C |w_n| e^{\beta k}n^{\alpha k} |\xi_n|^{\frac{n-k} n} \left|\frac z {e^{\beta} n^{\alpha}}\right|^{n-k}\\
&=
C \left|w_n \xi_n \left(\frac{z}{e^{\beta}n^{\alpha}}\right)^n\right|  (|\xi_n|^{\frac 1n} |z|)^{-k}.
\end{align*}
For every $z$ such that $|z|>\eps$ we obtain
\begin{align*}
\left|\sum_{k=1}^{n-1} w_{n-k}\xi_{n-k}\left(\frac z {e^{\beta} n^{\alpha}}\right)^{n-k}\right|
&\leq
C \left|w_n \xi_n \left(\frac{z}{e^{\beta}n^{\alpha}}\right)^n\right|
\cdot \left(\sum_{k=1}^{n-1} \frac 1{(3C+1)^{k}}\right)\\
&<
\frac 13 \left|w_n \xi_n \left(\frac{z}{e^{\beta}n^{\alpha}}\right)^n\right|.
\end{align*}
By~\eqref{eq:fk_asympt} and~\eqref{eq:no_log_moment_super_exp}, the right-hand side of this inequality goes to $+\infty$ as $n\to\infty$. For sufficiently large $n$, it is larger that $|\xi_0w_0|$. It follows that for $|z|>\eps$, the term of degree $n$ in the polynomial $\WWW_n(e^{-\beta}n^{-\alpha} z)$ is larger, in the sense of absolute value,  than the sum of all other terms.
Hence, the polynomial $\WWW_n(e^{-\beta}n^{-\alpha} z)$ has no zeros outside the disk $\DDD_{\eps}$.

\subsection{Proof of Theorem~\ref{theo:general_converse}} 
Start with a measure $\mu$ satisfying the assumptions of Theorem~\ref{theo:general_converse}. Define a function $I$ by $I(s)=\int_{-\infty}^{s}\mu(\DDD_{e^r})dr$ for $s<\log R_0$. The integral is finite by the second assumption of the theorem. Clearly, $I$ is non-decreasing, continuous and convex on $(-\infty, \log R_0)$. For $s>\log R_0$ let $I(s)=+\infty$. Define $I(\log R_0)$ by left continuity. Let now $u$ be defined as the Legendre--Fenchel transform of $I$:
$$
u(t)=\sup_{s\in\R} (st- I(s)).
$$
We claim that the random analytic function $\GGG_n(z)=\sum_{k=0}^{\infty} \xi_k f_{k,n} z^k$  with $f_{k,n}=e^{-n u(k/n)}$ satisfies assumptions (A1)--(A4) of Theorem~\ref{theo:general} with $f=e^{-u}$. By the Legendre--Fenchel duality, the function $u$ possesses the following properties. Firstly, it is convex and lower-semicontinuous. Secondly, it is finite on the interval $[0,T_0)$, where $T_0=\limsup_{t\to+\infty} I(t)/t$ satisfies $T_0\in (0,+\infty]$. This holds since  $I$ is non-decreasing and $\lim_{s\to -\infty} I(s)=0$ by construction.   Thirdly, $u(t)=+\infty$ for $t>T_0$ and $t<0$. This verifies assumption~(A1). Fourthly, formula~\eqref{eq:def_I} holds and $\lim_{t\to+\infty} u(t)/t=\log R_0$. This, together with Lemma~\ref{lem:log_moment}, shows that the convergence radius of $\GGG_n$ is $R_0$ a.s.\ and verifies assumption~(A4).  Finally, $u$ is continuous on $[0,T_0)$ (since it is convex and finite there), and, in the case $T_0<+\infty$, the function $u$ is left continuous  at $T_0$ (follows from the lower-semicontinuity of $u$). This verifies assumption~(A2). Assumption~(A3) holds trivially with $f=e^{-u}$.

\subsection{Proof of Proposition~\ref{prop:equilibrium}}
The logarithmic potential generated by the measure $\mu^{(\kappa)}$ is
$$
U^{(\kappa)}(z):=\int_{\C}\log \frac 1 {|z-w|}\mu^{(\kappa)}(dw).
$$
We will show that
\begin{equation}\label{eq:potential_mu_kappa}
U^{(\kappa)}(z)
= -
\begin{cases}
u(\kappa)-u(0), &|z|\leq e^{u'(0)},\\
I(\log |z|)+u(\kappa), &e^{u'(0)}\leq |z|\leq e^{u'(\kappa)}, \\
\kappa \log |z|, & |z|\geq e^{u'(\kappa)}.
\end{cases}
\end{equation}
Here, $u'(0)=\lim_{t\downarrow 0} u'(t)$. It follows from~\eqref{eq:potential_mu_kappa} that $F(z):=U^{(\kappa)}(z)+I(\log |z|)$ is constant and equal to $-u(\kappa)$ on the annulus $e^{u'(0)}\leq |z|\leq e^{u'(\kappa)}$ which contains the support of $\mu^{(\kappa)}$. For $|z|\geq e^{u'(\kappa)}$ we have $F(z)=I(\log |z|)-\kappa \log |z|\geq -u(\kappa)$ by~\eqref{eq:def_I}. Finally, for $|z|\leq e^{u'(0)}$ we have $F(z)=I(\log |z|)+u(0)-u(\kappa)\geq -u(\kappa)$ by~\eqref{eq:def_I}. To summarize, $F(z)$ is constant on the support of $\mu^{(\kappa)}$ and is at least as large as this constant outside the support.  By Theorem~3.3 on p.~44 of~\cite{saff_totik_book} this implies that $\mu^{(\kappa)}$ is the equilibrium measure for the logarithmic potential in the presence of external field $I(\log |z|)$. The theorem mentioned above is stated in~\cite{saff_totik_book} for $\kappa=1$ only, but it is valid for every $\kappa>0$.

It remains to prove~\eqref{eq:potential_mu_kappa}. The uniform probability distribution $\lambda_r$ on the boundary of $\DDD_r$ generates the potential
\begin{equation}\label{eq:potential_circle}
\int_{\partial \DDD_r} \log \frac 1 {|z-w|} \lambda_r(dw) = -\max \{ \log r, \log |z| \}.
\end{equation}
Write $F(r)=I'(\log r)$. Recall that $\mu^{(\kappa)}$ is a rotationally invariant measure of total mass $\kappa$ having support in the annulus $e^{u'(0)}\leq |z|\leq e^{u'(\kappa)}$ and such that $\mu^{(\kappa)}(\DDD_r)=F(r)$. Together with~\eqref{eq:potential_circle} this immediately implies~\eqref{eq:potential_mu_kappa} for $|z|\geq e^{u'(\kappa)}$. Suppose now that $|z|\leq e^{u'(0)}$. Using~\eqref{eq:potential_circle} and then integration by parts, we obtain
$$
-U^{(\kappa)}(z)
=
\int_{e^{u'(0)}}^{e^{u'(\kappa)}}\log |s| dF(s)
=
I'(r)r\Big|_{u'(0)}^{u'(\kappa)}-\int_{e^{u'(0)}}^{e^{u'(\kappa)}}I'(\log s)\frac{ds}{s}
=
u(\kappa)-u(0).
$$
Here, we used the identity $I(u'(t))=t u'(t)-u(t)$ for $t=0$ and $t=\kappa$. Suppose finally that $e^{u'(0)}\leq |z|\leq e^{u'(\kappa)}$. Using~\eqref{eq:potential_circle} we obtain
$$
-U^{(\kappa)}(z)
=
\int_{0}^{|z|} \log  |z| dF(s)+ \int_{|z|}^{e^{u'(\kappa)}}\log |s| dF(s)
=
I(\log |z|)+u(\kappa)
,
$$
where the first integral is equal to $F(z)\log |z|$ and integration by parts has been used for the second integral.
The proof of~\eqref{eq:potential_mu_kappa} is completed.

\bibliographystyle{plainnat}
\bibliography{littlewood_offord_bib}
\end{document}